\documentclass[a4paper,twopage,reqno,11pt]{amsart}
\usepackage[top=30mm,right=30mm,bottom=30mm,left=30mm]{geometry}

\newtheorem{theorem}{Theorem}[section]
\newtheorem{corollary}[theorem]{Corollary}

\newtheorem{lemma}[theorem]{Lemma}
\newtheorem{proposition}[theorem]{Proposition}

\theoremstyle{definition}

\newtheorem{example}[theorem]{Example}
\numberwithin{equation}{section}

\usepackage{amsmath,amssymb}
\usepackage{graphicx}
\usepackage{color}
\usepackage[colorlinks=true,citecolor=green,linkcolor=blue,
urlcolor=blue]{hyperref}
\usepackage{amsmath}
\usepackage{amssymb}
\usepackage{amsfonts}
\usepackage{amsthm}
\usepackage{amstext}
\usepackage{amsbsy}
\usepackage{amsopn}
\usepackage{amscd}
\usepackage{enumerate}
\usepackage{multirow}
\usepackage{lipsum}
\usepackage{rotating}
\usepackage{longtable}
\usepackage{float}
\usepackage{lscape}
\usepackage{pdflscape}
\usepackage[]{algorithm2e}
\usepackage{url}
\usepackage{tabularx}
\usepackage{xtab}

\newcommand{\GL}{\mathrm{GL}}
\newcommand{\SL}{\mathrm{SL}}

\newcommand{\Sp}{\mathrm{Sp}}

\newcommand{\SO}{\mathrm{SO}}
\newcommand{\SU}{\mathrm{SU}}
\renewcommand{\dim}{\mathrm{dim}}

\newcommand{\PSL}{\mathrm{PSL}}
\newcommand{\PSU}{\mathrm{PSU}}

\newcommand{\PSp}{\mathrm{PSp}}
\newcommand{\PGL}{\mathrm{PGL}}
\newcommand{\PGaL}{\mathrm{P\Gamma L}}

\newcommand{\AGaL}{A\Gamma L}
\newcommand{\POm}{\mathrm{P \Omega}}

\newcommand{\A}{\mathrm{Alt}}

\newcommand{\Alt}{\mathrm{Alt}}
\renewcommand{\S}{\mathrm{Sym}}
\newcommand{\Sym}{\mathrm{Sym}}

\newcommand{\Aut}{\mathrm{Aut}}

\newcommand{\Out}{\mathrm{Out}}

\newcommand{\PG}{\mathrm{PG}}

\newcommand{\Z}{\mathrm{Z}}

\newcommand{\Dmc}{\mathcal{D}}
\newcommand{\Bmc}{\mathcal{B}}

\newcommand{\Pmc}{\mathcal{P}}

\newcommand{\e}{\epsilon}

\renewcommand{\leq}{\leqslant}
\renewcommand{\geq}{\geqslant}

\renewcommand{\mod}[1]{\ (\mathrm{mod}{\ #1})}
\newcommand{\imod}[1]{\allowbreak\mkern4mu({\operator@font mod}\,\,#1)}

\begin{document}
	\title[]{Almost simple groups of Lie type and symmetric designs with $\lambda$ prime}
	
	\author[S.H. Alavi]{Seyed Hassan Alavi}
	\thanks{Corresponding author: S.H. Alavi}
	\address{Seyed Hassan Alavi, Department of Mathematics, Faculty of Science, Bu-Ali Sina University, Hamedan, Iran.}
	\email{alavi.s.hassan@basu.ac.ir and  alavi.s.hassan@gmail.com (G-mail is preferred)}
	\author[M. Bayat]{Mohsen Bayat}
	\address{Mohsen Bayat, Department of Mathematics, Faculty of Science, Bu-Ali Sina University, Hamedan, Iran.}
	\email{m.bayat@sci.basu.ac.ir}
	\author[A. Daneshkhah]{Asharf Daneshkhah}
	\address{Asharf Daneshkhah, Department of Mathematics, Faculty of Science, Bu-Ali Sina University, Hamedan, Iran.}
	\email{adanesh@basu.ac.ir}

	\date{\today}%

\begin{abstract}
  In this article, we investigate symmetric $(v,k,\lambda)$ designs $\Dmc$ with $\lambda$ prime admitting flag-transitive and point-primitive automorphism groups $G$. We prove that if $G$ is an almost simple group with socle a finite simple group of Lie type, then $\Dmc$ is either the point-hyperplane design of a projective space $\PG_{n-1}(q)$, or it is of parameters  $(7,4,2)$, $(11,5,2)$, $(11,6,2)$ or $(45,12,3)$.
\end{abstract}

\maketitle
\section{Introduction}\label{sec:intro}

A \emph{symmetric $(v,k,\lambda)$ design} is an incidence structure $\Dmc=(\Pmc,\Bmc)$  consisting of a set $\Pmc$ of $v$ \emph{points} and a set $\Bmc$ of $v$ \emph{blocks} such that every point is incident with exactly $k$ blocks, and every pair of blocks is incident with exactly $\lambda$ points. If $2<k<v-1$, then  $\Dmc$ is called a \emph{nontrivial} symmetric design.
A \emph{flag} of $\Dmc$ is an incident pair $(\alpha,B)$, where $\alpha$ and $B$ are a point and a block of $\Dmc$, respectively. An \emph{automorphism} of a symmetric design $\Dmc$ is a permutation of the points permuting the blocks and preserving the incidence relation. An automorphism group $G$ of $\Dmc$ is called \emph{flag-transitive} if it is transitive on the set of flags of $\Dmc$. If $G$ acts primitively on the point set $\Pmc$, then $G$ is said to be \emph{point-primitive}. A group $G$ is said to be \emph{almost simple} with socle $X$ if $X\unlhd G\leq \Aut(X)$, where $X$ is a nonabelian simple group. Further definitions and notation can be found in Section~\ref{sec:defn} below.

The main aim of this paper is to study symmetric designs with $\lambda$ prime admitting a flag-transitive and point-primitive almost simple automorphism group with socle being a finite simple groups of Lie type. Recently, we have studied nontrivial symmetric $(v, k, \lambda)$ design with prime $k$ admitting flag-transitive almost simple automorphism groups~\cite{a:ABCD-PrimeRep}, and proved that such a design is either a projective space, or it has a parameters set $(11, 5, 2)$.  We are now interested in possible classification of symmetric $(v, k, \lambda)$ designs $\Dmc$ with $\lambda$ prime admitting a flag-transitive and point-primitive almost simple automorphism group $G$. We have already shown in \cite{a:ABD-Exp} that almost simple exceptional groups of Lie type give rise to no possible symmetric designs with $\lambda$ prime. In the present paper, we focus on the case where $G$ is an almost simple group with socle $X$ being a finite simple classical group of Lie type, and prove that $\Dmc$ is either the point-hyperplane design of a projective space $\PG_{n-1}(q)$, or it is of parameters  $(7,4,2)$, $(11,5,2)$, $(11,6,2)$ or $(45,12,3)$, and we give detailed information of these designs in Section \ref{sec:examples}.

\begin{theorem}\label{thm:main}
    Let $\Dmc$ be a nontrivial symmetric $(v, k, \lambda)$ design with $\lambda$ prime, and let $\alpha$ be a point of $\Dmc$. If $G$ is a flag-transitive and point-primitive automorphism group of $\Dmc$ of almost simple group of Lie type with socle $X$. Then $\Dmc$ is the point-hyperplane design of $\PG_{n-1}(q)$ with $\lambda=(q^{n-2}-1)/(q-1)$ prime and $X=\PSL_{n}(q)$, or $\Dmc$ and $G$ are as in {\rm Table \ref{tbl:main}}.
\end{theorem}

Despite of the case where $k$ is prime, even in symmetric designs with $\lambda$ prime, flag-transitivity  does not necessarily imply point-primitivity. One of these examples arose from studying flag-transitive biplanes (symmetric designs with $\lambda=2$). It is known that there are only three non-isomorphic symmetric designs with parameters $(16,6,2)$, two of which admit flag-transitive and point-imprimitive design and one is not flag-transitive. The next interesting examples are the symmetric designs with parameters $(45,11,3)$. Indeed, Praeger \cite{a:Praeger-45-12-3} proves that there are only two examples of flag-transitive designs with parameters $(45,11,3)$. One is point-primitive and related to unitary geometry, while the other is point-imprimitive and constructed from a $1$-dimensional affine space for which we also  give an explicit base block in Section \ref{sec:examples} below. In general, Praeger and Zhou \cite{a:Praeger-imprimitive} study symmetric $(v,k,\lambda)$ designs admitting flag-transitive and point-imprimitive designs, and running through the potential parameters, we can only exclude one possibility, and so Corollary~\ref{cor:main} below is an immediate consequence of their result \cite[Theorem~1.1]{a:Praeger-imprimitive}.  To our knowledge, at this stage, any possible classification of flag-transitive and point-imprimitive designs with $\lambda$ prime seems to be out of reach.

\begin{table}
    \centering
    \caption{Parameters in Theorem~\ref{thm:main}}\label{tbl:main}
     \resizebox{\textwidth}{!}{
    \begin{tabular}{lllllllll}
        \noalign{\smallskip}\hline\noalign{\smallskip}
        Line &
        $v$ &
        $k$ &
        $\lambda$ &
        $X$ &
        $G$  &
        $G_{\alpha}$  &
        Designs &
        References$^{\ast}$ \\
        \noalign{\smallskip}\hline\noalign{\smallskip}
        $1$ &
        $7$ &
        $4$ &
        $2$ &
        $\PSL_{2}(7)$&
        $\PSL_{2}(7)$ &
        $\Sym_{4}$&
        Complement of Fano plane  &
        \cite{a:ABD-PSL2, b:Handbook} \\
        $2$ &
        $11$ &
        $5$ &
        $2$ &
        $\PSL_{2}(11)$&
        $\PSL_{2}(11)$ &
        $\Alt_{5}$&
        Hadamard &
        \cite{a:ABD-PSL2, b:Handbook} \\
        $3$ &
        $11$ &
        $6$ &
        $3$ &
        $\PSL_{2}(11)$&
        $\PSL_{2}(11)$ &
        $\Alt_{5}$&
        Complement of line $2$ &
        \cite{a:ABD-PSL2, b:Handbook} \\
        $4$ &
        $45$ &
        $12$ &
        $3$ &
        $\PSU_{4}(2)$  &
        $\PSU_{4}(2)$&
        $2{\cdot}(\Alt_{4}{\times} \Alt_{4}){\cdot} 2$ &
        -
        &\cite{a:Braic-2500-nopower,a:Dempwolff2001, a:Praeger-45-12-3}\\
        $5$ &
        $45$ &
        $12$ &
        $3$ &
        $\PSU_{4}(2)$  &
        $\PSU_{4}(2){:}2$&
        $2{\cdot}(\Alt_{4} {\times} \Alt_{4}).2{:}2$ &
        -
        &\cite{a:Braic-2500-nopower,a:Dempwolff2001, a:Praeger-45-12-3}\\
        \noalign{\smallskip}\hline\noalign{\smallskip}
        Note: &\multicolumn{8}{p{15cm}}{The last column addresses to references in which a design with the parameters in the line has been constructed.}
    \end{tabular}
}
\end{table}

\begin{corollary}\label{cor:main}
    Suppose that $\Dmc$ is a symmetric $(v,k,\lambda)$ design with $\lambda$ prime admitting flag-transitive and point-imprimitive automorphism group $G$. If $G$ leaves invariant a non-trivial partition $\mathcal{C}$ of $\Pmc$ with $d$ classes of size $c$, then there is a constant $l$ such that, for each $B\in \Bmc$ and $\Delta \in \mathcal{C}$,  $|B \cap\Delta|\in \{0,l\}$, and one of the following holds:
    \begin{enumerate}[\rm (a)]
        \item $k \leq \lambda(\lambda -3)/2$;
        \item $(v, k, \lambda) = (\lambda^2(\lambda +2), \lambda(\lambda + 1), \lambda)$ with $(c, d, l) = (\lambda^2,\lambda +2,\lambda)$ or $(\lambda +2,\lambda^2, 2)$;
        \item $(v, k, \lambda, c, d, l) = ((\lambda + 6)(\lambda^2+4\lambda-1)/4 , \lambda (\lambda+5)/2 , \lambda, \lambda+6, (\lambda^2+4\lambda-1)/4,3)$, where $\lambda\equiv 1$ or $3$ $\mod 6$.
    \end{enumerate}
\end{corollary}

\subsection{Outline of proofs}\label{sec:outline}
In order to prove Theorem~\ref{thm:main} in Section~\ref{sec:Proof}, as noted above, by \cite[Corollary~1.2]{a:ADM-AS-CP}, we only need to consider the case where the socle $X$ of $G$ is a finite simple classical group. In particular, by \cite{a:ABD-PSL3,a:ABD-PSL4,a:ABD-PSL2,a:ABDZ-PSU4,a:DZ-PSU3},  in the case where $X$ is a linear or unitary group, we can assume that the  dimension of the underlying vector space is at least $5$. Moreover, we include all possible symmetric $(v,k,\lambda)$ designs for $\lambda=2,3$ obtained in \cite{a:Zhou-lam3-affine,t:Regueiro,a:Regueiro-classification} and therein references, and so we can also assume that $\lambda\geq 5$. If $\lambda$ is coprime to $k$, then the possible designs can be read off from  \cite[Corollary~1.2]{a:ABD-Exp}. Since $\lambda(v-1)=k(k-1)$,
we need to focus on the case where $\lambda$ divides $k$.
Since also $G$ is point-primitive, a point-stabiliser $H=G_\alpha$ is maximal in $G$. Note that $v=|G:H|$ is odd as $\lambda$ is odd prime and $\lambda(v-1)=k(k-1)$. Therefore, as a key tool, we use a classification of primitive permutation groups of odd degree \cite[Theorem]{a:Liebeck-Odd} which gives
the possible candidates for $H$. Another important and useful fact is that $k$ divides the order of $H$, and  so $\lambda$ is a prime divisor of $|H|$. At some stage, the knowledge of subdegrees (length of suborbits) of the $G$-action on the right cosets of $H$ in $G$ is essential. We now analyse each possibilities of $H$. Considering the fact that $k$ divides $\lambda(v-1)$ and if applicable $k$ also divides $\lambda d$ with $d$ a subdegree, we find a polynomial $f(q)$ of smallest possible degree for which $k$ divides $\lambda f(q)$. As $\lambda$ is a odd prime divisor of $|H|$, we find possible upper bounds $u_{\lambda}$. In most cases, we observe that $v< u_{\lambda} f(q)^2$ does not hold and this violates the fact that $\lambda v<k^2$. In some cases, the inequality $v< u_{\lambda} f(q)^2$ has some solutions, and these solutions suggest some parameters set that are needed to be argued as well. In the remaining cases, however, we need to use some other arguments and new techniques to settle down our claims. In this manner, Theorem \ref{thm:main} follows from Propositions~\ref{prop:psl}-\ref{prop:orth}.
The proof of Corollary~\ref{cor:main} is also given in Section~\ref{sec:Proof}, and the proof follows immediately from  \cite[Theorem~1.1]{a:Praeger-imprimitive} by ruling out one possible case.
In this paper, we use the software \textsf{GAP} \cite{GAP4} for computational arguments.

\subsection{Definitions and notation}\label{sec:defn}

All groups and incidence structures in this paper are finite. Symmetric and alternating groups on $n$ letters are denoted by $\S_{n}$ and $\A_{n}$, respectively.  We write ``$n$'' for a group of order $n$. Also for a given positive integer $n$ and a prime divisor $p$ of $n$, we denote the $p$-part of $n$ by $n_{p}$, that is to say, $n_{p}=p^{t}$ with $p^{t}\mid n$ but $p^{t+1}\nmid n$. For finite simple groups of Lie type, we adopt the standard notation as in \cite{b:Atlas}, and in particular, we use the standard notation to denote the finite simple classical groups, that is to say, $\PSL_{n}(q)$, for $n\geq 2$ and $(n,q)\neq (2,2), (2,3)$,
$\PSU_{n}(q)$,  for  $n\geq 3$  and  $(n,q)\neq (3,2)$,
$\PSp_{2m}(q)$,  for $n=2 m \geq 4$  and  $(m,q)\neq (2,2)$,
$\Omega_{2m+1}(q)=\POm_{2m+1}(q)$,  for  $n=2m+1\geq 7$  and  $q$  odd,
$\POm_{2m}^{\pm}(q)$,   for $n=2m\geq 8$. In this manner, the only repetitions are
\begin{align*}
\nonumber  &\PSL_{2}(4)\cong \PSL_{2}(5)\cong \A_{5}, & &
\PSL_{2}(7)\cong \PSL_{3}(2), & &
\PSL_{2}(9)\cong \A_{6},\\
&\PSL_{4}(2)\cong \A_{8}, & &\PSp_{4}(3)\cong \PSU_{4}(2).
\end{align*}

Recall that a symmetric design $\Dmc$ with parameters  $(v, k, \lambda)$ is a pair $(\Pmc,\Bmc)$, where $\Pmc$ is a set of $v$ points and $\Bmc$ is a set of $v$ blocks such that each block is a $k$-subset of $\Pmc$ and each two distinct points are contained in $\lambda$ blocks. We say that $\Dmc$ is nontrivial if $2 < k < v-1$.
Further notation and definitions in both design theory and group theory are standard and can be found, for example in~\cite{b:Beth-I,b:Atlas,b:Dixon,t:Kleidman,b:Lander}.

\section{Examples and Comments}\label{sec:examples}

In this section, we provide some examples of symmetric designs with $\lambda$ prime admitting a flag-transitive automorphism almost simple group with socle $X$. We remark here that the designs in Table~\ref{tbl:main} can be found in \cite{a:ABD-PSL2, a:ABDZ-PSU4}, but the construction given here is obtained by \textsf{ GAP} \cite{GAP4}.\smallskip

\begin{example}\label{ex:proj-space}
    The point-hyperplane of a projective space $\PG_{n-1}(q)$ with parameters  $((q^{n}-1)/(q-1),(q^{n-1}-1)/(q-1),(q^{n-2}-1)/(q-1))$ for $n\geq 3$ is a well-known example of flag-transitive symmetric designs. Any group $G$ with $\PSL_{n}(q) \leq G \leq \PGaL_{n}(q)$ acts flag-transitively on $\PG_{n-1}(q)$. If $n=3$, then we have the \emph{Desargusian plane} with parameters $(q^{2}+q+1,q+1,1)$ which is a projective plane. The design $\Dmc$ with parameters $(7, 4, 2)$ in line $1$ of Table~\ref{tbl:main} is the complement of the unique well-known symmetric design, namely, \emph{Fano Plane} admitting flag-transitive and point-primitive automorphism group $\PSL_{2}(7)\cong \PSU_{2}(7)$ with point-stabiliser $\Sym_{4}$.
\end{example}

\begin{example}\label{ex:hadamard}
    The symmetric $(11,5, 2)$ design is a \emph{Paley difference set} which is also a Hadamard design with the base block $\{ 1, 2, 3, 5, 11 \}$, and its full automorphism group is $\PSU_{2}(11)$ acting  flag-transitively and point-primitively. In this case, the point-stabiliser is isomorphic to $\Alt_{5}$. The complement of this design is the unique symmetric $(11,6,3)$ design whose full automorphism group $\PSU_{2}(11)$ is also flag-transitive and point-primitive with $\Alt_{5}$ as point-stabiliser.
\end{example}

\begin{example}\label{ex:16-6-2}
    There are exactly three non-isomorphic symmetric $(16,6,2)$ design, two of which are flag-transitive.
    The first symmetric design admitting a flag-transitive automorphism group is constructed from a difference set in $2^4$ whose automorphism group is $2^4~\Sym_6< 2^4~\GL_4(2)$ with point-stabiliser $\Sym_6$. The second example of  symmetric $(16,6,2)$ design admitting a flag-transitive automorphism group  arose from a difference set in $\Z_2\times \Z_8$ , and the point-stabiliser of order $48$ acts as the full group of symmetries of the cube, hence is a central extension $\Sym_4\circ 2$ of the symmetric group $\Sym_4$ by a group of order $2$. These two designs admit point-imprimitive automorphism group. The last symmetric $(16,6,2)$ design can be constructed as a difference set in $Q_{8} \times \Z_2$. The full automorphism group of order $16\cdot 24$ of this design is not flag-transitive.
\end{example}

\begin{example}\label{ex:other}
    Mathon and Spence \cite{a:Mathon-96} have constructed $3,752$ pairwise non-isomorphic symmetric $(45,12,3)$ designs, and they have shown that at least $1,136$ of these designs have a trivial automorphism group. Cheryl E. Praeger in \cite{a:Praeger-45-12-3} constructs two flag-transitive symmetric $(45,12,3)$ designs, and proves that these designs are the only two examples. One of these symmetric designs is related to unitary
    geometry and admits point-primitive automorphism group $\PSU_{4}(2)\cdot 2$, while the other has  point-imprimitive  automorphism group $G\leq {\rm \AGaL}_1(81)$. The base block of the former design is $\{ 1, 2, 4, 5, 12, 15, 17, 21, 28, 34,$ $35, 38 \}$, and more detailed information about this design can be found in \cite{a:Braic-2500-nopower,b:Handbook} and therein references. We here give an explicit base block for the point-imprimitive example. Let $G$ be a permutation group on the set $\Pmc:=\{1,\ldots,45\}$ generated by the permutations $\sigma_1$, \ldots, $\sigma_5$ below\smallskip
    
    \small
    \begin{tabular}{p{.5cm}p{13cm}}
        $\sigma_1:=$&$(1,2,4,5,3)( 6,16,43,13,14)( 7,39,33,45,26)( 8,21,37,32,28)( 9,11,25,35,10)$
        $( 12,44,24,40,17) (15,30,38,23,19) ( 18,34,20,31,41) ( 22,36,27,42,29)$, \\
        $\sigma_2:=$&$(1,5,2,3,4)( 6,10,16,9,43,11,13,25,14,35)( 7,40,39,17,33,12,45,44,26,24)$ $( 8,23,21,19,37,15,32,30,28,38)( 18,22,34,36,20,27,31,42,41,29)$, \\
        $\sigma_3:=$&$(2,5,3,4)( 6,17,32,20,11,26,23,29)( 7,30,42,43,12,21,34,35)$ $( 8,31,10,45,15,22,13,
        40)( 9,39,19,27,14,44,28,18)( 16,24,37,41,25,33,38,36)$, \\
        $\sigma_4:=$&$(2,3) ( 4,5) ( 6,32,11,23) ( 7,42,12,34) ( 8,10,15,13) ( 9,19,14,28)( 16,37,
        25,38)$ $( 17,20,26,29)( 18,39,27,44) ( 21,35,30,43) ( 22,40,31,45) ( 24,41,33,36)$, \\
        $\sigma_5:=$&$(1,6,11) ( 3,40,45) ( 4,41,36) ( 5,13,10) ( 8,35,
        39) ( 9,42,38) ( 14,37,34) ( 15,44,43)$ 
        $( 17,32,29)( 18,30,33) ( 20,23,26) ( 21,27,24)$. \\
    \end{tabular}
    \normalsize
    \smallskip
    
    \noindent Then $G\cong 3^4:(5:8)$ is isomorphic to a subgroup of ${\rm \AGaL}_1(81)$. The group $G$ has a subgroup $K\cong 3^2:8$ with an orbit of size $12$, namely, $B=\{1, 2, 3, 4, 6, 11, 19, 28, 36, 40, 41, 45\}$. Let now $\Bmc$ be the set of $G$-orbits $B^G$. Then $\Dmc=(\Pmc,\Bmc)$ forms a symmetric $(45,12,3)$ design with flag-transitive automorphism group $G$. Moreover, $C=\{1, 6, 11, 17, 20, 23, 26, 29, 32 \}$ is a $G$-invariant partition on $\Pmc$, and so $G$ is point-imprimitive. Note that the full automorphism group of $\Dmc$ is isomorphic to $3^4:(\SL_{2}(5):2)$ which is also point-imprimitive.
\end{example}

\section{Preliminaries}\label{sec:pre}

In this section, we state some useful facts in both design theory and group theory. Recall that a group $G$ is called almost simple if $X\unlhd G\leq \Aut(X)$, where $X$ is a (nonabelian) simple group.

\begin{lemma}\label{lem:New}{\rm \cite[Lemma 2.2]{a:ABD-PSL3}}
    Let $G$  be an almost simple group with socle $X$, and let $H$ be maximal in $G$ not containing $X$. Then $G=HX$ and $|H|$ divides $|\Out(X)|\cdot |X\cap H|$.
\end{lemma}

\begin{lemma}\label{lem:six}{\rm \cite[Lemma 2.1]{a:ABD-PSL2}}
    Let $\Dmc$ be a symmetric $(v,k,\lambda)$ design, and let $G$ be a flag-transitive automorphism group of $\Dmc$. If $\alpha$ is a point of $\Dmc$ and $H=G_{\alpha}$, then
    \begin{enumerate}[\rm \quad (a)]
        \item $k(k-1)=\lambda(v-1)$;
        \item $k\mid |H|$ and $\lambda v<k^2$;
        \item $k\mid \lambda d$, for all nontrivial subdegrees $d$ of $G$.
    \end{enumerate}
\end{lemma}

\begin{lemma}[Tits' Lemma]\label{lem:Tits1}{\rm \cite[1.6]{a:Seitz-TitsLemma}}
    If $X$ is a group of Lie type in characteristic $p$, then any proper subgroup of index prime to $p$ is contained in a proper parabolic
    subgroup of $X$.
\end{lemma}

\begin{lemma}\label{lem:Tits}{\rm \cite[1.6]{a:Seitz-TitsLemma}}
    Suppose that $\Dmc$ is a symmetric $(v,k,\lambda)$ design admitting a flag-transitive and point-primitive almost simple automorphism group $G$ with socle $X$ of Lie type in odd characteristic $p$. Suppose also that the point-stabiliser $G_{\alpha}$, not containing $X$, is not a parabolic subgroup of $G$. Then $\gcd(p,v-1)=1$.
\end{lemma}

If a group $G$ acts on a set $\Pmc$ and $\alpha\in \Pmc$, the \emph{subdegrees} of $G$ are the size of orbits of the action of the point-stabiliser $G_\alpha$ on $\Pmc$.

\begin{lemma}\label{lem:subdeg}{\rm \cite{a:LSS1987}}
    If $X$ is a group of Lie type in characteristic $p$, acting on the set of cosets of a maximal parabolic subgroup, and $X$ is not $\PSL_{n}(q)$, $\POm^{+}_{n}(q)$(with $n/2$ odd) and $E_{6}(q)$, then there is a unique subdegree which is a power of $p$.
\end{lemma}

For a point stabiliser $H$ of an automorphisms group $G$ of a flag-transitive design $\Dmc$, by Lemma~\ref{lem:six}(b), we conclude that $\lambda|G|\leq |H|^{3}$, and so we have that

\begin{corollary}\label{cor:large}
    Let $\Dmc$ be a flag-transitive $(v, k, \lambda)$ symmetric design with automorphism group $G$. Then $|G|\leq |G_{\alpha}|^3$,  where $\alpha$ is a point in $\Dmc$, and so $|X|<|\Out(X)|^2{\cdot}|H\cap X|^3$.
\end{corollary}

\begin{lemma}\label{lem:divisible}{\rm \cite[Lemma 2.5]{a:ABCD-PrimeRep}}
    Suppose that $\Dmc$ is a $(v, k, \lambda)$ symmetric design. Let $G$ be a flag-transitive automorphism group of $\Dmc$ with simple socle $X$ of Lie type in characteristic $p$. If the point-stabiliser $H=G_{\alpha}$ contains a normal quasi-simple subgroup $N$ of Lie type in characteristic $p$ and $p$ does not divide $|Z(N)|$, then either $p$ divides $k$, or $N_{B}$ is contained in a parabolic subgroup $P$ of $N$ and $k$ is divisible by $|N{:}P|$.
\end{lemma}

The following result gives a classification of primitive groups of odd degree of almost simple type with socle finite simple classical groups. This result is proved independently in \cite{a:Kantor-87-Odd} and \cite{a:Liebeck-Odd}. Here we follow the description of this groups as in \cite{a:Liebeck-Odd}.

\begin{lemma}{\rm \cite[Theorem]{a:Liebeck-Odd}}\label{lem:odd-deg}
    Let $G$ be a primitive permutation group of odd degree $v$ on the set $\Gamma$. Assume that the socle $X=X(q)$ of $G$ is a simple classical group with a natural projective module $V=V_{n}(q)$, where $q=p^a$ and $p$ prime, and let $H=G_{\alpha}$ be the stabilizer of a point $\alpha \in \Gamma$, then one of the following holds:
    \begin{enumerate}[\rm(a)]
        \item  if $q$ is odd then one of $(i),(ii)$ below holds:
        \begin{enumerate}[\rm(i)]
            \item $X$ is a classical group with natural projective module $V=V_{n}(q)$ and one of $(1)$-$(7)$ below holds:
            \begin{enumerate}[\rm(1)]
                \item $H$ is the stabilizer of a nonsingular subspace (any subspace for $X=\PSL_{n}(q)$);
                \item $H\cap X$ is the stabilizer of an orthogonal decomposition $V=\oplus V_{j}$ with all $V_{j}$'s isometric (any decomposition $V=\oplus V_{j}$ with $\dim(V_{j})$ constant for $X=\PSL_{n}(q)$);
                \item $X=\PSL_{n}(q)$, $H$ is the stabilizer of a pair $\{U,W\}$ of subspaces of complementary dimensions with $U\leq W$ or $U\oplus W=V$, and G contains a graph automorphism;
                \item $H\cap X$ is $\SO_{7}(2)$ or $\Omega_{8}^{+}(2)$ and $X$ is $\Omega_{7}(q)$ or $\POm_{8}^{+}(q)$, respectively, $q$ is prime and $q\equiv \pm 3 \mod{8}$;
                \item $X=\POm_{8}^{+}(q)$, $q$ is prime and $q\equiv \pm 3 \mod{8}$, $G$ contains a triality automorphism of $X$ and $H\cap X$ is $2^3\cdot 2^6\cdot \PSL_{3}(2)$;
                \item $X=\PSL_{2}(q)$ and $H\cap X$ is dihedral, $\Alt_{4}$, $\Sym_{4}$, $\Alt_{5}$ or $\PGL_{2}(q_{0})$, where $q=q_{0}^2$;
                \item $X=\PSU_{3}(5)$ and $H\cap X=M_{10}$.
            \end{enumerate}
            \item  $H = N_{G}(X(q_{0}))$, where $q=q_{0}^t$ and $t$ is an odd prime;
        \end{enumerate}
        \item  if $q$ is even then $H\cap X$ is a parabolic subgroup of $X$.
    \end{enumerate}
\end{lemma}

We will use the following results in order to obtain suitable lower or upper bounds for parameters of possible designs. The proof of these results can be found in \cite{a:ABD-Un-CP,a:AB-Large-15}

\begin{lemma}{\rm\cite[Lemma~4.2 and Corollary~4.3]{a:AB-Large-15}}\label{lem:up-lo-b}
    \begin{enumerate}[\rm(a)]
        \item  If $n\geq 2$, then
        \begin{align*}
        q^{n^2-2}<&|\PSL_{n}(q)|\leq |\SL_{n}(q)|<(1-q^{-2})q^{n^2-1}\\
        (1-q^{-1})q^{n^2-2}<&|\PSU_{n}(q)|\leq |\SU_{n}(q)|<(1-q^{-2})(1+q^{-3})q^{n^2-1}
        \end{align*}
        \item  If $n\geq 4$, then
        \begin{align*}
        \frac{1}{4}q^{n(n-1)/2}<&|\Omega_{n}(q)|<|\SO_{n}(q)|\leq (1-q^{-2})(1-q^{-4})q^{n(n-1)/2}\\
        \frac{1}{2\beta}q^{n(n+1)/2}<&|\PSp_{n}(q)|\leq |\Sp_{n}(q)|\leq (1-q^{-2})(1-q^{-4})q^{n(n+1)/2}
        \end{align*}
        with $\beta=\gcd(2, q-1)$.
        \item  If $n\geq 6$, then
        \begin{align*}
        \frac{1}{8}q^{n(n-1)/2}<& |\POm^{\pm}_{n}(q)|<|\SO^{\pm}_{n}(q)|\leq \delta (1-q^{-2})(1-q^{-4})(1+q^{-n/2})q^{n(n-1)/2}
        \end{align*}
        with $\delta=\gcd(2, q)$.
    \end{enumerate}
\end{lemma}

\begin{lemma}{\rm\cite[Lemma~4.4]{a:AB-Large-15}}\label{lem:factor-bound}
    Suppose that $t$ is a positive integer. Then
    \begin{enumerate}[\rm (a)]
        \item if $t\geq 5$, then $t!<5^{(t^2-3t+1)/3}$;
        \item if $t\geq 4$, then $t!<2^{4t(t-3)/3}$.
    \end{enumerate}
\end{lemma}

\begin{lemma}\label{lem:equation}{\rm \cite[Lemma 3.12]{a:ABD-Un-CP}}
    Let $q$ be a prime power and $n\geq 3$ be a positive integer number, then
    \begin{align*}
    q^{\frac{n(n-1)}{2}}<\prod_{j=2}^{n}(q^{j}-1)<\prod_{j=2}^{n}(q^{j}-(-1)^{j})<q^{\frac{n^2+n-2}{2}}.
    \end{align*}
\end{lemma}

\section{Proof of the main results}\label{sec:Proof}

In this section, we prove Theorem \ref{thm:main} and Corollary~\ref{cor:main}. Suppose that $\Dmc$ is a nontrivial symmetric design with $\lambda$ prime, and that $G$ is an automorphism group of $\Dmc$ which is an almost simple group whose socle $X$ is a finite nonabelian simple group of Lie type. Suppose now that $G$ is flag-transitive and point-primitive. Let $H=G_{\alpha}$, where $\alpha$ is a point of $\Dmc$. Then $H$ is maximal in $G$ (see \cite[7, Corollary 1.5A]{b:Dixon}), and so Lemma~\ref{lem:New} implies that
\begin{align}
v=\frac{|X|}{|H \cap X|}.\label{eq:v}
\end{align}

As mentioned in Section \ref{sec:outline}, we only need  to focus on the case where $X$ is a finite simple classical group. Moreover, the parameter $v$ is odd and the possibilities for $H$ can be read off from \cite{a:Liebeck-Odd} which are also recorded in Lemma \ref{lem:odd-deg}. Further, we can assume that $\lambda\geq 5$ is an odd prime and in the case where $X$ is $\PSL_{n}(q)$ or $\PSU_{n}(q)$, we can also assume that $n\geq 5$. In Propositions~\ref{prop:psl}-\ref{prop:orth} below, we discuss possible cases for the pairs $(X,H)$, and finally prove Theorem \ref{thm:main}.
In what follows, we denote by $\,^{\hat{}}H$ the pre-image of the group $H$ in the corresponding group.

\begin{proposition}\label{prop:psl}
    Let $\Dmc$ be a nontrivial symmetric $(v,k,\lambda)$ design with $\lambda\geq 5$ prime. Suppose that $G$ is an automorphism group of $\Dmc$ of almost simple type with socle $X=\PSL_{n}(q)$ for $ n\geq 5$. If $G$ is flag-transitive and $H=G_{\alpha}$ with $\alpha$ a point of $\Dmc$, then $\Dmc$ is the point-hyperplane design of $\PG_{n-1}(q)$ with $\lambda=(q^{n-2}-1)/(q-1)$ prime and $H\cap X\cong \,^{\hat{}}[q^{n-1}]{:}\SL_{n-1}(q){\cdot} (q-1)$.
\end{proposition}
\begin{proof}
    Let $H_{0}=H\cap X$, where $H=G_{\alpha}$ with $\alpha$ a point of $\Dmc$. It follows from Lemma \ref{lem:six}(a) that $v$ is odd. Then by Lemma \ref{lem:odd-deg}, we have one of the following possibilities:
    \begin{enumerate}[\rm (1)]
        \item $H_{0}$ is a parabolic subgroup of $X$;
        \item $H$ is the stabilizer of a pair $\{U,W\}$ of subspaces of complementary dimensions with $U\leq W$ and $G$ contains a graph automorphism.
        \item $q$ is odd, and $H$ is the stabilizer of a pair $\{U,W\}$ of subspaces of complementary dimensions with $U\oplus W=V$, and $G$ contains a graph automorphism.
        \item $q$ is odd, and $H_{0}$ is the stabilizer of a partition $V=V_{1} \oplus \cdots \oplus V_{t}$ with $\dim (V_{j})=i$;
        \item  $q=q_{0}^t$ is odd with $t$ odd prime, and $H = N_{G}(X(q_{0}))$;
    \end{enumerate}
    In what follows, we analyse each of these possible cases separately. \smallskip
    
    \noindent \textbf{(1)} Let $H_{0}$ be a parabolic subgroup of $X$. In this case, $H=P_{i}$, where $i\leq \lfloor n/2 \rfloor$, and by~\cite[Proposition 4.1.17]{b:KL-90}, the subgroup $H_{0}$ is isomorphic to
    \begin{align*}
    \,^{\hat{}}q^{i(n-i)}:\SL_{i}(q)\times \SL_{n-i}(q)\cdot(q-1).
    \end{align*}
    Suppose first that $H =P_{1}$. Then $G$ is $2$-transitive, and this case has already been studied by Kantor \cite{a:Kantor-85-2-trans}. Therefore,  $\Dmc$ is the point-hyperplane design of $\PG_{n-1}(q)$ with parameters set $((q^{n}-1)/(q-1), (q^{n-1}-1)/(q-1), (q^{n-2}-1)/(q-1))$ and $\lambda=(q^{n-2}-1)/(q-1)$ prime, as desired.
    
    Suppose now that $H=P_{i}$ with $i\geq 2$. It follows from \eqref{eq:v} and \cite[p. 534]{a:Regueiro-classical} that
    \begin{align}\label{eq:A-ii-1-L-v-1}
    v=\frac{(q^{n}-1)(q^{n-1}-1)\cdots(q^{n-i+1}-1)}{(q^{i}-1)\cdots(q^2-1)(q-1)}>q^{i(n-i)}.
    \end{align}
    Then by Lemmas \ref{lem:New} and \ref{lem:six}(b), the parameter $k$ divides $|\Out(X)|\cdot|H_{0}|$, where $|H_{0}|=q^{n(n-1)/2}\gcd(n, q-1)^{-1}\cdot \prod_{j=2}^{n-i}(q^{j}-1)\cdot\prod_{j=1}^{i}(q^{j}-1)$ and $|\Out(X)|=2a \cdot \gcd(n, q-1)$. Note that $\lambda$ is an odd prime divisor of $k$. Then $\lambda$ must divide $a$, $p$, $q-1$ or $(q^{j}-1)/(q-1)$, for some $j\in \{2, \ldots, n-i\}$, and so
    \begin{align}\label{eq:lam-A-ii-1-L-1}
    \lambda\leq (q^{n-i}-1)/(q-1).
    \end{align}
    Here by Lemma \ref{lem:six}(c) and \cite[Corollary 2]{a:Korableva-LnUn}, the parameter $k$ divides $\lambda d_{i,j}(q)$, where
    \begin{equation}\label{eq:Sub-A-ii-L}
    d_{i, j}(q)=q^{j^2}\cdot \prod_{l=i-j+1}^{i}(q^{l}-1)\cdot \prod_{l=n-j-i+1}^{n-i}(q^{l}-1)\cdot \prod_{l=1}^{j}(q^{l}-1)^{-2},
    \end{equation}
    for $l=1, 2,\ldots, i$. Therefore, $k$ divides $\lambda d_{i,1}(q)$, where $d_{i,1}(q)=q(q^i-1)(q^{n-i}-1)(q-1)^{-2}$. Then by \eqref{eq:A-ii-1-L-v-1} and Lemma~\ref{lem:six}(b), we have that
    \begin{align*}
    \lambda q^{i(n-i)}<\lambda v< k^{2}\leq \lambda^2 q^2(q^i-1)^2(q^{n-i}-1)^2(q-1)^{-4}.
    \end{align*}
    Thus $q^{i(n-i)}\cdot (q-1)^4 <\lambda q^2(q^i-1)^2(q^{n-i}-1)^2$, and so \eqref{eq:lam-A-ii-1-L-1} implies that $q^{i(n-i)}(q-1)^5< q^2(q^i-1)^2(q^{n-i}-1)^3<q^{3n-i+2}$. Thus
    \begin{align}\label{eq:psl-parab-1}
    q^{i(n-i)}(q-1)^5<q^{3n-i+2},
    \end{align}
    and hence $n(i-3)<i^2-i+2$. Note that $2i\leq n$. Thus  $2i(i-3)\leq n(i-3) <i^2-i-1$, and so $i^2<5i+2$. Hence $i=2, 3, 4, 5$.
    
    If $i=5$, then by \eqref{eq:psl-parab-1}, we have that $q^{2n-22}(q-1)^5<1$. Since $n\geq 2i=10$, the last inequality holds only for $(n,q)=(10, 2)$, in which case by \eqref{eq:A-ii-1-L-v-1}, $v=109221651$. Moreover, by Lemmas \ref{lem:New} and \ref{lem:six}(b), $k$ divides $|\Out(X)|\cdot|H_{0}|$. Thus $k$ is a divisor of $6710027434028590694400$. It is easy to check that for possible $k$, the fraction $k(k-1)/(v-1)$ is not a prime number.
    
    If $i=4$, then \eqref{eq:psl-parab-1} implies that $q^{n-14}(q-1)^5<1$, and so  $n\in \{8,9,10,11,12,13,14\}$ as $n\geq 2i=8$. Note by \eqref{eq:A-ii-1-L-v-1} that $q$ is even as $v$ is odd. Then $\gcd(v-1,q^{2}+1)=1$. Recall by Lemma~\ref{lem:six} that $k$ divides $\lambda \gcd(v-1, d_{4,1})$. Then $v<\lambda\cdot [d_{4,1}/(q^{2}+1)]^{2}$, where $\lambda\leq (q^{n-4}-1)/(q-1)$,  and hence
    \begin{align}\label{eq:parab-i=4}
    v<(q-1)^{-1}(q^{n-4}-1)\cdot [d_{4,1}/(q^{2}+1)]^{2}.
    \end{align}
    For each possible $n$, by straightforward calculation,  we observe that  \eqref{eq:parab-i=4} does not hold.
    
    If $i=2$, then $G$ is a rank $3$ primitive group, see \cite{a:Kantor-rank3}.  The symmetric designs admitting primitive rank $3$ automorphism groups have been classified by Dempwolff \cite{a:Dempwolff2001}. Running through all these possible cases, we can not find any such symmetric design with $\lambda\geq 5$ prime.

    If $i=3$, then \eqref{eq:A-ii-1-L-v-1} implies that
    \begin{align*}
    v=\frac{(q^n-1)(q^{n-1}-1)(q^{n-2}-1)}{(q^3-1)(q^2-1)(q-1)}>q^{3n-9}.
    \end{align*}
    We now consider the following cases:\smallskip
    
    \noindent \textbf{(1.1)} Let $q$ be odd. If $n$ is even, then $v$ is also even, which is impossible. Therefore, $n$ is odd. Note by \eqref{eq:Sub-A-ii-L},  \cite[p. 338]{a:Saxl2002} and Lemma \ref{lem:six}(c) that $k$ divides $\lambda \gcd(d_{3,1}(q), d_{3,2}(q))$, where $d_{3,1}(q)=q(q^{2}+q+1)(q^{n-3}-1)(q-1)^{-1}$ and $d_{3,2}(q)=q^{4}(q^{2}+q+1)(q^{n-4}-1)(q^{n-3}-1)(q-1)^{-1}(q^{2}-1)^{-1}$. Therefore,
    \begin{align}\label{eq:k-A-ii-L}
    k  \text{ divides } \lambda f(q),
    \end{align}
    where $f(q)=q(q^2+q+1)(q^{n-3}-1)(q^2-1)^{-1}$. Then by Lemma \ref{lem:six}(b), we have that $\lambda q^{3n-9}<\lambda v<k^2\leq \lambda^2 q^2(q^2+q+1)^2(q^{n-3}-1)^2(q^2-1)^{-2}$. Thus
    \begin{align}\label{eq:A-ii-L}
    q^{3n-9}(q^2-1)^2<\lambda q^2(q^2+q+1)^2(q^{n-3}-1)^2.
    \end{align}
    Since $\lambda$ is an odd prime divisor of $k$, Lemmas \ref{lem:New} and \ref{lem:six}(b) imply that $\lambda$  divides $a$, $p$ , $q-1$ or $(q^{j}-1)(q-1)^{-1}$, for some $j\in\{2, 3, \ldots, n-3\}$.
    
    Suppose first that $\lambda$ divides $a$, $p$ or $q-1$. Then by \eqref{eq:A-ii-L}, we have that $q^{3n-9}(q^2-1)^2<q^3(q^2+q+1)^2(q^{n-3}-1)^2$, and so $q^{3n-9}(q^2-1)^2<q^{2n-3}(q^2+q+1)^2$. Hence $q^{n-6}(q^2-1)^2<(q^2+q+1)^2$. Since $(q^2+q+1)^2<q(q^2-1)^2$, we conclude that $q^{n-6}<q$, which is impossible as  $n\geq 2i=6$ is odd.
    
    Suppose now that
    \begin{align}\label{eq:lam-A-ii-L}
    \lambda \quad \text{divides} \quad (q^{j}-1)(q-1)^{-1},
    \end{align}
    for some $j\in \{2, 3, \ldots, n-3\}$. Since $q^{n-3}-1<q^{n-3}$ and $q^{j}-1<q^{j}$, it follows from \eqref{eq:A-ii-L} that $q^{3n-9}(q-1)(q^2-1)^2<q^{2n+j-4}(q^2+q+1)^2$, and so $q^{n-j-5}<(q^2+q+1)^2/[(q-1)(q^2-1)^2]$. As $(q^2+q+1)^2<q(q-1)(q^2-1)^2$, we conclude that $q^{n-j-6}<1$, and so $j>n-6$.  Since $j\leq n-3$, we have that $j\in\{n-5, n-4, n-3\}$, where $n$ is odd. We now consider the following two subcases.\smallskip
    
    \noindent \textbf{(1.1.1)} Let $j=n-3$ or $n-5$. Note that $j$ is even and $\lambda$ divides $q^{j}-1$ by \eqref{eq:lam-A-ii-L}. Since $\lambda$ is prime, it follows that $\lambda \leq q^{(n-3)/2}+1$, and so  \eqref{eq:A-ii-L} yields $q^{n-9}<(q^2+q+1)^4(q^2-1)^{-4}$. Since $(q^2+q+1)^4<q^2(q^2-1)^4$, we have that $q^{n-9}<q^2$, or equivalently, $q^{n-11}<1$. Since also $n>6$ is odd, we conclude that $n=7, 9, 11$. Then by \eqref{eq:A-ii-1-L-v-1}, we can obtain $v$. Note for these parameters $v$ that $\gcd(v-1,q^{2}+q+1)=1$. Since by Lemma~\ref{lem:six}, the parameter $k$ divides $\lambda \gcd(v-1, d_{3,1})$, we conclude by \eqref{eq:lam-A-ii-L} that $v<(q-1)^{-1}(q^{n-3}-1)\cdot [d_{3,1}/(q^{2}+q+1)]^{2}$, but for each possible $n$, this inequality does not hold for $q\geq 3$.\smallskip

    \noindent \textbf{(1.1.2)} Let $j=n-4$. Then by \eqref{eq:lam-A-ii-L}, the parameter $\lambda$ divides $(q^{n-4}-1)(q-1)^{-1}$. Let $u$ be a positive integer  such that $\lambda u=(q^{n-4}-1)(q-1)^{-1}$. Note that $(q^{n-4}-1)(q-1)^{-1}$ is odd, and so $u$ is an odd number. Here by \eqref{eq:A-ii-L} and \eqref{eq:lam-A-ii-L}, $uq^{3n-9}(q-1)(q^2-1)<q^2(q^2+q+1)^2(q^{n-3}-1)^2(q^{n-4}-1)<(q^2+q+1)^2 q^{3n-8}$, and so $u\cdot (q-1)(q^2-1)^2<q(q^2+q+1)^2$. This inequality holds only for $u=1$ or $(u,q)=(3,3)$. In the latter case, since $\lambda u=(q^{n-4}-1)(q-1)^{-1}$, it follows that $u=3$ divides $q^{n-5}+q^{n-6}+\ldots+q+1$, where $q=3$, which is impossible. Therefore, $u=1$, and hence $\lambda=q^{n-5}+q^{n-6}+\ldots+q+1$. Thus  by \eqref{eq:k-A-ii-L}, the parameter $k$ divides $\lambda f(q)$, where $f(q)=q(q^2+q+1)(q^{n-3}-1)(q^2-1)^{-1}$. Let now $m$ be a positive integer such that $mk=\lambda f(q)$. Then by Lemma \ref{lem:six}(a), we have that $k=1+m\cdot (v-1)/f(q)$. Note by \eqref{eq:A-ii-1-L-v-1} that $v-1=g(q)\cdot q(q^{n-3}-1)(q^3-1)^{-1}(q^2-1)^{-1}(q-1)^{-1}$, where $g(q)=q^{2n-1}+q^{n+2}-q^{n+1}-q^{n}-q^{n-1}+q^5-q^4-q^3+q+1$. Therefore, $k=1+[m\cdot g(q)/(q^3-1)^2]>mq^{2n-7}$. Since $k$ divides $q(q^2+q+1)(q^{n-4}-1)(q^{n-3}-1)(q^2-1)^{-1}(q-1)^{-1}$, we conclude that
    \begin{align*}
    m\cdot q^{2n-7}<q(q^2+q+1)(q^{n-4}-1)(q^{n-3}-1)(q^2-1)^{-1}(q-1)^{-1},
    \end{align*}
    and so $m\cdot (q-1)(q^2-1)<q(q^2+q+1)$. This inequality holds only for $m=1, 2$. Let now $r(q)=(q+1)(q^3-1)^2$ and $h(q)=q^{n+3}+q^{7}+q^{6}-q^{5}-q^{4}-q^{3}$. Then
    \begin{align}\label{lem31-1}
    (q^{3}-1)^2\cdot k=m\cdot g(q)+(q^{3}-1)^2=m\cdot h(q)(q^{n-4}-1)+m\cdot r(q)+(q^{3}-1)^2.
    \end{align}
    Since $\lambda=(q^{n-4}-1)/(q-1)$ is an odd prime divisor of $k=[m\cdot g(q)+(q^{3}-1)^2]/(q^{3}-1)^{2}$, it follows from \eqref{lem31-1} that $q^{n-4}-1$ divides $m\cdot r(q)+(q^{3}-1)^2=[m\cdot (q+1)+1](q^3-1)^2$. Recall that $m\leq 2$. Therefore, $q^{n-4}-1\leq [m\cdot (q+1)+1](q^3-1)^2\leq (2q+3)(q^3-1)^2$. Since $q$ is odd, we conclude that $2q+3\leq q^2$, and so $q^{n-4}-1\leq q^{8}$. This inequality holds only for $n=7, 9, 11$. By the same manner as in the previous cases, we observe that $v-1$ is coprime to $q^{2}+q+1$, and since $\lambda\leq (q-1)^{-1}(q^{n-4}-1)$, it follows that $v<(q-1)^{-1}(q^{n-4}-1)\cdot [d_{3,1}/(q^{2}+q+1)]^{2}$ for $n\in \{7,9,11\}$. But for each possible $n$, this inequality does not hold for any $q\geq 3$.\smallskip
    
    \noindent \textbf{(1.2)} Let $q$ be even. Then \eqref{eq:Sub-A-ii-L} and Lemma \ref{lem:six}(c) imply that
    \begin{align}\label{eq:k-B-L}
    k \text{ divides } \lambda f(q),
    \end{align}
    where $f(q):=d_{3,1}=q(q^2+q+1)(q^{n-3}-1)/(q-1)$. Note that $\lambda$ is an odd prime divisor of $a$ or $q^{j}-1$ with $j\in\{1,\ldots, n-3\}$.\smallskip
    
    Suppose first that $\lambda$ divides $a$ or $q-1$. Then by Lemma \ref{lem:six}(b), we have that $\lambda v<k^2\leq \lambda^2 q^2(q^2+q+1)^2(q^{n-3}-1)^2/(q-1)^2$. Thus $q^{3n-9}(q-1)^2<\lambda q^2(q^2+q+1)^2(q^{n-3}-1)^2<q^3(q^2+q+1)^2(q^{n-3}-1)^2$, and so $q^{3n-9}(q-1)^2<q^{2n-3}(q^2+q+1)^2$. Hence $q^{n-6}(q-1)^2<(q^2+q+1)^2$. Since $(q^2+q+1)^2<q^6(q-1)^2$, we conclude that $q^{n-6}<q^{6}$, and so $n=6,7, 8,9, 10$. Define
    \begin{align*}
    d_{n}(q)=\left\{
    \begin{array}{ll}
    f(q)/(q^{2}+q+1)^{2} , & \hbox{if $n=6,9$;} \\
    f(q)/(q^{2}+q+1) , & \hbox{if $n=7,8,10$.}
    \end{array}
    \right.
    \end{align*}
    Note that $\lambda\leq q-1$. Then by the same manner as before, we must have $v<(q-1)\cdot d_{n}(q)^{2}$. By solving this inequality for $n\in\{6,\ldots,10\}$, we conclude that $q=2$ when $n\in\{7,8,10\}$. In these cases, however, $\lambda\leq \max\{a,q-1\}=1$, which is a contradiction.\smallskip
    
    Suppose now that $\lambda$ divides $q^{j}-1$, for some $j\in \{2,\ldots, n-3\}$. Therefore, $\lambda \leq q^{j}-1$. By \eqref{eq:k-B-L} and Lemma \ref{lem:six}(b), we have that $\lambda v<k^2\leq \lambda^2 q^2(q^2+q+1)^2(q^{n-3}-1)^2/(q-1)^2$. Thus
    \begin{align}\label{eq:eq-B-L-2}
    q^{3n-9}(q-1)^2<\lambda q^2(q^2+q+1)^2(q^{n-3}-1)^2.
    \end{align}
    Recall that $\lambda \leq q^{j}-1$. Hence $q^{3n-9}(q-1)^2<(q^2+q+1)^2 q^{2n+j-4}$. Since $(q^2+q+1)^2<q^6(q-1)^2$, we conclude that $q^{3n-9}<q^{2n+j+2}$, and so $j>n-11$.  Since $j\leq n-3$, we have that $j\in\{n-10,n-9,\ldots,n-3\}$. Recall that $\lambda$ divides $q^{j}-1$. Let $u$ be a positive integer such that
    \begin{align}\label{eq:B-L-lam}
    \lambda=\frac{q^{j}-1}{u}.
    \end{align}
    Let now $m$ be a positive integer such that $mk=\lambda f_{n}(q)$ , where
    \begin{align*}
    f_{n}(q)=q(q^2+q+1)(q^{n-3}-1)/(q-1).
    \end{align*}
    Then by Lemma \ref{lem:six}(a), we have that $k=1+m(v-1)/f_{n}(q)$. Note that $v-1=q(q^{n-3}-1)g_{n}(q)(q^3-1)^{-1}(q^2-1)^{-1}(q-1)^{-1}$, where $g_{n}(q)=q^{2n-1}+q^{n+2}-q^{n+1}-q^{n}-q^{n-1}+q^5-q^4-q^3+q+1$. Therefore,
    \begin{align}\label{eq:B-L-k}
    k=1+\frac{m\cdot g_{n}(q)}{(q+1)(q^3-1)^2}>m\cdot q^{2n-9}.
    \end{align}
    Since $k\leq \lambda f_{n}(q)$ and $\lambda=(q^{j}-1)/u$, we conclude that $mq^{2n-9}<k\leq q(q^2+q+1)(q^{j}-1)(q^{n-3}-1)/[u\cdot (q-1)]$, and so $mu\cdot q^{n-j-7}(q-1)<(q^2+q+1)$. Hence,
    \begin{align}\label{eq:B-L-1-mu}
    mu<q^{j-n+7}.
    \end{align}
    Since $mu\geq 1$, it follows  that $q^{j-n+7}>1$, and so $j-n+7>0$. Therefore,
    \begin{align}\label{eq:parab-j}
    j=n-t \text{ with $t\in\{3,\ldots,6\}$.}
    \end{align}
    
    Let $h_{j}(q)$ and $r_{j}(q)$ be as in the second and third columns of Table \ref{tbl:B-L-1}. Then $g_{n}(q)=h_{j}(q)\cdot (q^{j}-1)+r_{j}(q)$, and so
    \begin{align}
    \nonumber    (q+1)(q^{3}-1)^2\cdot k&=m\cdot g_{n}(q)+(q+1)(q^{3}-1)^2\\
    &=m\cdot h(q)(q^{j}-1)+m\cdot r(q)+(q+1)(q^{3}-1)^2.\label{eq:h-d-g}
    \end{align}
    For $j$ as in \eqref{eq:parab-j}, we observe that $|m\cdot r_j(q)+(q^{3}-1)^2(q+1)|>0$.  Since $\lambda=(q^{j}-1)/u$ is a divisor of $k$, it follows from \eqref{eq:h-d-g} that $(q^{j}-1)/u$ divides $|m\cdot r_j(q)+(q+1)(q^{3}-1)^2|$, where $mu< q^{j-n+7}$ and $r_j(q)$ is as in Table~\ref{tbl:B-L-1}.
    Since $|r_j(q)+(q+1)(q^{3}-1)^2|<q^{10}$, we have that $q^{j}-1<muq^{10}$, and so by \eqref{eq:B-L-1-mu}, we have that $q^{n-17}<1$. This inequality holds only for $n=6,\ldots, 16$. Define
    \begin{align*}
    d_{n}(q)=\left\{
    \begin{array}{ll}
    9\cdot f_{n}(q)/(q^{2}+q+1)^{2} , & \hbox{if $n=6,9,12,15$;} \\
    3\cdot f_{n}(q)/(q^{2}+q+1) , & \hbox{if $n=7,8,10,11,13,14,16$.}
    \end{array}
    \right.
    \end{align*}
    Note that $\lambda\leq q^{n-3}-1$. Since $k$ divides $\lambda\cdot \gcd(v-1,f_{n}(q))$ which is a divisor of $\lambda d_{n}(q)$, the inequality $\lambda v<k^{2}$ implies that $v<(q^{n-3}-1)\cdot d_{n}(q)^{2}$, and considering each possible $n$, we conclude that $q\in\{2,4,8\}$. For each $q$, we obtain the parameter $v$ by \eqref{eq:A-ii-1-L-v-1}, and considering all divisors $k$ of $|\Out(X)|\cdot |H_{0}|$, we observe that $k(k-1)/(v-1)$ is not prime, which is a contradiction.\smallskip
    
    \begin{table}
        \centering
        \small
        \caption{The polynomials $h(q)$ and $r(q)$ as in Case 1.2 of Proposition~\ref{prop:psl}.}\label{tbl:B-L-1}
        \begin{tabular}{cll}
            \noalign{\smallskip}\hline\noalign{\smallskip}
            $j$ &
            $h_{j}(q)$ &
            $r_{j}(q)$
            \\
            \noalign{\smallskip}\hline\noalign{\smallskip}
            $n-3$ &
            $q^{n+2}+2q^{5}-q^{4}-q^{3}-q^{2} $ &
            $ 3q^{5}-2q^{4}-2q^{3}-q^{2}+q+1 $
            \\
            $n-4$ &
            $ q^{n+3}+q^{7}+q^{6}-q^{5}-q^{4}-q^{3} $ &
            $ q^{7}+q^{6}-2q^{4}-2q^{3}+q+1 $
            \\
            $n-5$ &
            $ q^{n+4}+q^{9}+q^{7}-q^{6}-q^{5}-q^{4} $ &
            $ q^{9}+q^{7}-q^{6}-2q^{4}-q^{3}+q+1 $
            \\
            $n-6$ &
            $ q^{n+5}+q^{11}+q^{8}-q^{7}-q^{6}-q^{5}+q^{2} $ &
            $ q^{8}-q^{7}-q^{6}-q^{4}-q^{3}+q^{2}+q+1 $
            \\
            \noalign{\smallskip}\hline\noalign{\smallskip}
        \end{tabular}
    \end{table}

    \noindent \textbf{(2)} Let $H$ be the stabilizer of a pair $\{U,W\}$ of subspaces of dimension $i$ and $n-i$ with $2i<n$ and $U\leq W$.
    Then by~\cite[Proposition 4.1.22]{b:KL-90}, the subgroup $H_{0}$ is isomorphic to $\,^{\hat{}}[q^{2in-3i^2}]\cdot \SL_{i}(q)^2\times \SL_{n-2i}(q)\cdot  (q-1)^2$. It follows from~\eqref{eq:v} and Lemma~\ref{lem:up-lo-b} that $v>q^{i(2n-3i)}$. We note here that Lemma~\ref{lem:subdeg} is still true in this case. Then there is a subdegree which is a power of $p$. On the other hand, if $p$ is odd, then the $p$-part $(v-1)_p$ of $v-1$ is $q$. Then by Lemma \ref{lem:six}(c), $k$ divides $\lambda q$. Hence  Lemma \ref{lem:six}(b) implies that $\lambda q^{i(2n-3i)}<\lambda v<k^2\leq \lambda^2 q^{2}$, and so
    \begin{align}\label{eq:eq-A-ii-3-1}
    q^{i(2n-3i)}<\lambda q^2.
    \end{align}
    Note that $\lambda$ is an odd prime divisor of $2a\cdot \gcd(n, q-1)\cdot |H_{0}|$. It follows from Lemmas \ref{lem:New} and \ref{lem:six}(b) that $\lambda$ divides $a$, $p$ or $q^{j}-1$, where $j\leq n-2$. Then  $\lambda \leq q^{n-2}-1$, and so \eqref{eq:eq-A-ii-3-1} implies that $q^{i(2n-3i)}<q^2(q^{n-2}-1)$. Thus $n(2i-1)<3i^2$. Since $n>2i$, we have that $i^2<2i$. This inequality holds only for $i=1$, in which case by \eqref{eq:eq-A-ii-3-1}, we conclude that $q^{2n-3i}<\lambda q^2$, where $\lambda \leq q^{n-2}-1$. Then $q^{2n-3}<q^{n}$, and so $n<3$, which is a contradiction.\smallskip
    
    \noindent \textbf{(3)} Let $H$ be the stabilizer of a pair $\{U,W\}$ of subspaces of dimension $i$ and $n-i$ with $2i<n$ and $V=U\oplus W$. Then by~\cite[Proposition 4.1.4]{b:KL-90}, the subgroup $H_{0}$ is isomorphic to $\,^{\hat{}}\SL_{i}(q)\times \SL_{n-i}(q)\cdot (q-1)$. We first show that $i\neq 1$. If $i=1$, then by \eqref{eq:v}, we have that $v=q^{n-1}(q^{n}-1)/(q-1)$. Note by \cite[p. 339]{a:Saxl2002} that $k$ divides $\lambda q^{n-2}(q^{n-1}-1)/(q-1)$. On the other hand, by Lemmas \ref{lem:six}(a) and~\ref{lem:Tits}, the  parameter $k$ divides $\lambda (v-1)$ and $v-1$ is coprime to $q$. Thus
    \begin{align}\label{eq:k-A-ii-3-2}
    k \text{ divides } \lambda (q^{n-1}-1)/(q-1).
    \end{align}
    We now apply Lemma \ref{lem:six}(b) and conclude that $q^{n-1}(q^{n}-1)<\lambda (q^{n-1}-1)^2/(q-1)$. Note that $\lambda$ is an odd prime divisor of $2a\cdot \gcd(n, q-1)\cdot |H_{0}|$. Then Lemmas \ref{lem:New} and \ref{lem:six}(b) imply that $\lambda$ divides $a$, $p$ or $q^{j}-1$, with $j\leq n-1$. If $\lambda$ divides $a$, $p$ or $q-1$, then the inequality $q^{n-1}(q^{n}-1)<\lambda (q^{n-1}-1)^2/(q-1)$ yields $q^{n-1}(q^{n}-1)(q-1)<q(q^{n-1}-1)^2$, which is impossible. Therefore,
    \begin{align}\label{eq:lam-A-ii-3-2}
    \lambda \text{ divides } (q^{j}-1)/(q-1),
    \end{align}
    for some $j\in\{2, \ldots, n-1\}$. By \eqref{eq:k-A-ii-3-2}, $k$ divides $\lambda (q^{n-1}-1)/(q-1)$. Let $u$ be a positive integer such that $uk=\lambda f_{n}(q)$, where $f_{n}(q)=(q^{n-1}-1)/(q-1)$. Since $v-1=(q^{n-1}-1)(q^{n}+q-1)/(q-1)$, by Lemma \ref{lem:six}(a), we have that
    \begin{align}\label{eq:k-lam-A-ii-2}
    k=u\cdot (q^{n}+q-1)+1 \text{ and }  \lambda=u^2q(q-1)+\frac{u^2(2q-1)+u}{f_{n}(q)}.
    \end{align}
    Recall that $k$ divides $\lambda f_{n}(q)$. So \eqref{eq:k-lam-A-ii-2} implies that $u\cdot (q^{n}+q-1)+1$ divides $(q^{j}-1)(q^{n-1}-1)/(q-1)^2$. Since
    \begin{align*}
    u\cdot (q^{j}-1)(q^{n-1}-1)=q^{j-1}[u\cdot (q^{n}+q-1)+1]-u\cdot (q^{n-1}+2q^{j}-q^{j-1}-1)-q^{j-1},
    \end{align*}
    we have that $u\cdot (q^{n}+q-1)+1$ divides $u\cdot (q^{n-1}+2q^{j}-q^{j-1}-1)+q^{j-1}$. Thus $u\cdot (q^{n}+q-1)+1\leq u\cdot (q^{n-1}+2q^{j}-q^{j-1}-1)+q^{j-1}$, and so $q^{n}+q\leq q^{n-1}+2q^{j}$. Note that $j\leq n-1$. Then $q^{n}+q\leq 3q^{n-1}$, which is impossible. Therefore, $i\geq 2$. In this case, by~\cite[p. 340]{a:Saxl2002}, we have that $v>q^{2i(n-i)}$.  It follows from \cite[p. 339-340]{a:Saxl2002} that $k$ divides $\lambda (q^{i}-1)(q^{n-i}-1)$. Hence by Lemma \ref{lem:six}(b), we have that $\lambda q^{2i(n-i)}<\lambda v<k^2\leq \lambda^2(q^{i}-1)^2(q^{n-i}-1)^2$, and so
    \begin{align}\label{eq:eq-A-ii-3-2}
    q^{2i(n-i)}<\lambda q^{2n}.
    \end{align}
    Note that $\lambda$ is an odd prime divisor of $k$ dividing $2a\cdot \gcd(n, q-1)\cdot |H_{0}|$. Then Lemmas \ref{lem:New} and \ref{lem:six}(b) imply that $\lambda$ is a divisor of $a$, $p$ or $q^{j}-1$, for some $j\leq n-i$. Thus
    \begin{align}\label{eq:lam-A-ii-3-2-1}
    \lambda \leq (q^{n-i}-1)/(q-1),
    \end{align}
    and by \eqref{eq:eq-A-ii-3-2}, we have that $q^{2i(n-i)}(q-1)<q^{2n}(q^{n-i}-1)$. Therefore, $2i(n-i)<3n-i$, and hence $n(2i-3)<2i^2-i$. This implies that $(n,i)=(5,2)$, in which case $v=q^6(q^5-1)(q^2+1)/(q-1)$ and $k$ divides $\lambda (q^2-1)(q^3-1)$. Then by Lemma \ref{lem:six}(a), the parameter $k$ divides $\lambda \gcd(v-1, (q^2-1)(q^3-1))$. Since $\gcd(v-1, q+1)=1$, we conclude that $k$  divides $\lambda (q-1)^2(q^2+q+1)$. Then the inequality $\lambda v<k^2$ and \eqref{eq:lam-A-ii-3-2-1} yields $q^6(q^5-1)(q^2+1)<(q-1)^2(q^{3}-1)^3$, which is impossible.\smallskip
    
    \noindent \textbf{(4)} Here $V=V_{1} \oplus \cdots \oplus V_{t}$ with $\dim (V_{j})=i$ and $n=it$. By~\cite[Proposition 4.2.9]{b:KL-90}, the subgroup $H_{0}$ is isomorphic to $\,^{\hat{}}\SL_{i}(q)^t\cdot (q-1)^{t-1}\cdot  \Sym_{t}$. It follows from \cite[p.12]{a:ABCD-PrimeRep} that $v>q^{n(n-i)}/(t!)$. Let $i=1$. By~\cite[p. 340]{a:Saxl2002}, we have that $k$ divides $2 \lambda n(n-1)(q-1)$. Then Lemma \ref{lem:six}(b) implies that $\lambda q^{n(n-1)}/(n!)<\lambda v<k^2\leq \lambda^2 4n^2(n-1)^2(q-1)^2$. Therefore,
    \begin{align}\label{eq:eq1-A-iii-L}
    q^{n(n-1)}<4\lambda \cdot (n!)\cdot n^2(n-1)^2(q-1)^2.
    \end{align}
    Since $\lambda$ is an odd prime divisor of $k$, by Lemmas \ref{lem:New} and \ref{lem:six}(b), $\lambda$ must divide $a$, $n!$ or $q-1$. Then $\lambda\leq \max\{a, n, q-1\}$, and so $\lambda<n\cdot (q-1) $. Thus by \eqref{eq:eq1-A-iii-L}, we conclude that
    \begin{align}\label{eq:A-ii-2-2-SL}
    q^{n(n-1)}<4n^3\cdot (n!)\cdot (n-2)^2(q-1)^3.
    \end{align}
    It follows from Lemma \ref{lem:factor-bound} that $q^{n(n-1)}<2^{[4n(n-3)+6]/3}\cdot  n^5(q-1)^3$. Since $n^5<2^{3n}$, we conclude that $q^{3n^2-3n-9}<2^{4n^2-3n+6}$, and so $3n^2-3n-9\leq (3n^2-3n-9)\cdot \log_{p}q \leq (4n^2-3n+6)\cdot \log_{p}2\leq (4n^2-3n+6)\cdot \log_{3}2< (4n^2-3n+6)\times  0.7$. Hence $2n^2-9n<132$. This inequality holds only for $n=6, 7, 8, 9, 10$. However, for each such value of $n$, the inequality \eqref{eq:A-ii-2-2-SL} does not hold, which is a contradiction. Therefore, $i\geq 2$, in which case by~\cite[p. 340]{a:Saxl2002}, $k$ must divide $\lambda t(t-1)(q^i-1)^2(q-1)^{-1}$. Then Lemma \ref{lem:six}(b) implies that $\lambda q^{n(n-i)}/(t!)<\lambda v<k^2\leq \lambda ^2 t^2(t-1)^2(q^i-1)^4(q-1)^{-2}$. Therefore,
    \begin{align}\label{eq:eq2-A-iii-L}
    q^{n(n-i)}\cdot (q-1)^2<\lambda\cdot  (t!)\cdot t^2\cdot (t-1)^2(q^i-1)^4.
    \end{align}
    Since $\lambda$ is an odd prime divisor of $k$, by Lemmas \ref{lem:New} and \ref{lem:six}(b), $\lambda$ must divide $a$, $p$ , $t!$ or $q^{j}-1$ for some $j\leq i$, and so $\lambda\leq \max\{a, p, t, (q^{i}-1)(q-1)^{-1}\}$, consequently
    \begin{align}\label{eq:lam-A-iii-2-SL}
    \lambda<t\cdot (q^{i}-1)(q-1)^{-1}.
    \end{align}
    Then \eqref{eq:eq2-A-iii-L} implies that $q^{n(n-i)}\cdot (q-1)^3<t^5\cdot (t!)\cdot (q^i-1)^5$. If $t\geq 4$, then by Lemma \ref{lem:factor-bound}(b), we have that $t!<2^{4t(t-3)/3}$, and hence  $q^{3n(n-i)}\cdot (q-1)^{9}<2^{4t(t-3)}\cdot t^{15}\cdot (q^{i}-1)^{15}$. Since $t^{15}<2^{9t}$ and $q^{i}-1<q^{i}$, it follows that
    \begin{align}\label{eq:A-iii-2-3-SL}
    q^{3n^2-3i(n+5)+3}<2^{4t^2-3t},
    \end{align}
    where $n=it$. Therefore, $t^2(3i^2-4)+3<3t(i^2-1)+15i<3ti(i+5)$, and so $t(3i^2-4)<3i(i+5)$. This inequality holds only for $(i, t)\in\{(2, 4), (2, 5)\}$. For these pairs of $(i, t)$, we can easily observe that the inequality \eqref{eq:A-iii-2-3-SL} does not hold, which is a contradiction. Hence $t=2, 3$. If $t=2$, then \eqref{eq:eq2-A-iii-L} and \eqref{eq:lam-A-iii-2-SL} imply that $q^{2i^2}\cdot (q-1)^3<16\cdot (q^i-1)^5$. As  $(q-1)^3\geq 8$, we conclude that $q^{2i^2-5i-1}<1$, and so $i=2$ for which $n=2i=4$, which is impossible. If $t=3$, then by \eqref{eq:eq2-A-iii-L} and \eqref{eq:lam-A-iii-2-SL}, we have that $q^{6i^2}\cdot (q-1)^3<2^3\cdot 3^4\cdot (q^i-1)^5$, and so $q^{6i^2}<q^{5i+4}$, which is impossible.\smallskip

    \noindent \textbf{(5)} Let $H = N_{G}(X(q_{0}))$ with $q=q_{0}^t$ odd and $t$ odd prime. Then  by~\cite[Proposition 4.5.3]{b:KL-90}, the subgroup $H_{0}$ is isomorphic to
    \begin{align*}
    \,^{\hat{}}\SL_{n}(q_{0})\cdot \gcd((q-1)(q_{0}-1)^{-1}, n)
    \end{align*}
    with $q=q_{0}^t$. Note that $|\Out(X)|=2a\cdot \gcd(n, q-1)$. Since $|X|<|\Out(X)|^2\cdot |H_{0}|^3$ by Corollary \ref{cor:large}, it follows from Lemma~\ref{lem:up-lo-b} that $q_{0}^{t(n^2-2)}<4a^2\cdot q_{0}^{3n^2}(q_{0}^t-1)^3$. As $a^2<2q$, we have that $q_{0}^{n^2(t-3)-6t}<8$. Since also $q_{0}$ is odd, it follows that $3^{n^2(t-3)-6t}\leq q_{0}^{n^2(t-3)-6t}<8<3^2$, and so $3^{n^2(t-3)-6t}<3^2$. Therefore, $t(n^2-6)<3n^2+2$. If $t\geq 5$, then $5(n^2-6)\leq t(n^2-6)<3n^2+2$, and so $n^2<16$, which is impossible as $n\geq 5$. Therefore, $t=3$. In this case by \eqref{eq:v} and Lemma \ref{lem:up-lo-b}, we conclude that $v>q_{0}^{2n^2-9}$. It follows from Lemma \ref{lem:six}(a)-(c) that $k$ divides $\lambda \gcd(v-1,|\Out(X)|\cdot |H_{0}|)$. By Tits' Lemma~\ref{lem:Tits} $v-1$ is coprime to $q_{0}$, and so $k$ must divide $2\lambda a\cdot  g(q_{0})$, where $g(q_{0})=(q_{0}^{n}-1) \cdots (q_{0}^{2}-1)\cdot \gcd(q_{0}^2+q_{0}+1, n)$. Then by Lemma \ref{lem:six}(b), we have that $\lambda q_{0}^{2n^2-9}< \lambda v <k^{2}\leq 4a^2\lambda^2 (q_{0}^{n}-1)^2{\cdots}(q_{0}^{2}-1)^2\cdot (q_{0}^2+q_{0}+1)^2$. Note that $(q_{0}^2+q_{0}+1)^2<q_{0}^5$. So
    \begin{align}\label{eq:eq-A-i-L}
    q_{0}^{n^2-n-12}< 4\lambda a^2.
    \end{align}
    Note by Lemmas \ref{lem:New} and \ref{lem:six}(b) that $\lambda$ is  an odd prime divisor of $a$, $p$, $q_{0}-1$ or $(q_{0}^{j}-1)/(q_{0}-1)$ with $j\in \{2, 3, \ldots, n\}$, and so $\lambda \leq (q_{0}^{n}-1)/(q_{0}-1)$. Then by the inequality \eqref{eq:eq-A-i-L}, we conclude that $q_{0}^{n^2-2n-12}(q_{0}-1)<4a^2$. Recall that $a=ts=3s$. Then $q_{0}^{n^2-2n-12}(q_{0}-1)<36s^2$. Since $n\geq 5$, we have that $q_{0}^3(q_{0}-1)\leq q_{0}^{n^2-2n-12}(q_{0}-1)<36\cdot s^2$, and hence $q_{0}^3(q_{0}-1)<36\cdot s^2$, which is impossible.
\end{proof}

\begin{proposition}\label{prop:unitary}
    Let $\Dmc$ be a nontrivial symmetric $(v,k,\lambda)$ design with $\lambda\geq 5 $ prime. Suppose that $G$ is an automorphism group of $\Dmc$ of almost simple type with socle $X$. If $G$ is flag-transitive, then the socle $X$ cannot be $\PSU_{n}(q)$ with $n\geq 5$.
\end{proposition}
\begin{proof}
    Let $H_{0}=H\cap X$, where $H=G_{\alpha}$ with $\alpha$ a point of $\Dmc$.  Then by Lemma \ref{lem:six}(a), the parameter $v$ is odd, and so by Lemma \ref{lem:odd-deg}, one of the following holds:
    \begin{enumerate}[\rm (1)]
        \item $q$ is even, and $H_{0}$ is a parabolic subgroup of $X$;
        \item $q$ is odd, and $H$ is the stabilizer of a nonsingular subspace;
        \item $q$ is odd, and $H_{0}$ is the stabilizer of an orthogonal decomposition $V=\oplus V_{j}$ with all $V_{j}$'s isometric;
        \item $q=q_{0}^t$ is odd with $t$ odd prime, and $H = N_{G}(X(q_{0}))$.
    \end{enumerate}
    We analyse each of these possible cases separately and arrive at a contradiction in each case. \smallskip
    
    \noindent \textbf{(1)} Let $H_{0}$ be a parabolic subgroup of $X$. Note in this case that $q=2^a$ is even. By~\cite[Proposition 4.1.18]{b:KL-90}, the subgroup $H_{0}$ is isomorphic to $\,^{\hat{}}q^{i(2n-3i)}:\SL_{i}(q^2)\cdot\SU_{n-2i}(q)\cdot(q^2-1)$, for $i\leq \lfloor n/2\rfloor$. It follows from~\eqref{eq:v} and \cite{a:Saxl2002} that $v>q^{i(2n-3i)}$. By Lemma~\ref{lem:subdeg}, there is a unique subdegree $d=2^{c}$. Note that
    \begin{align*}
    (v-1)_{2}=\left\{
    \begin{array}{ll}
    q,& \hbox{if $n$ is even and $i = n/2$;} \\
    q^{3}, & \hbox{if $n$ is odd and $i = (n-1)/2$;} \\
    q^{2}, & \hbox{otherwise,}
    \end{array}
    \right.
    \end{align*}
    where $(v-1)_{2}$ is the $2$-part of $v-1$. Since $k$  divides $\lambda \gcd(v-1, d )$, it follows that $k$ divides $\lambda q^t$, where $t=1, 2, 3$. Note that $\lambda$ is an odd prime divisor of $k$. It follows from Lemma \ref{lem:six}(b) that $\lambda$ must divide $a$, $q^{2j}-1$, for some $j \in \{1, 2, \ldots, i\}$ or $q^{j}-1$, for some $j\in\{ 2,\ldots, n-2i\}$. Since $\max\{q^{n-2i}+1, q^{i}+1\}<q^{n-2}+1$, we conclude that
    \begin{align}\label{eq:lam-B-U}
    \lambda <q^{n-2}+1,
    \end{align}
    where $2i\leq n$. Then by Lemma \ref{lem:six}(b), we have that $\lambda q^{i(2n-3i)}<\lambda v<k^2\leq \lambda^2 q^{6}$, and so
    \begin{align}\label{eq:eq-B-U}
    q^{i(2n-3i)}<\lambda q^{6}.
    \end{align}
    It follows from \eqref{eq:lam-B-U}  that  $q^{i(2n-3i)}<q^{6}(q^{n-2}+1)$. Since $q^6(q^{n-2}+1)<q^{n+5}$, we have that $q^{i(2n-3i)}<q^{n+5}$, and so
    \begin{align}\label{eq:Un-par-n}
    n\cdot (2i-1)<3i^2+5.
    \end{align}
    As $n\geq 2i$, it follows that $i^2<2i+5$. This inequality holds only for $i=1, 2, 3$. If $i=1$, then $k$ divides $\lambda q^{2}$. Let $u$ be a positive integer such that $uk=\lambda q^2$. Since $\lambda<k$, we have that $u<q^{2}$. By \cite[Lemma 3.7(a)]{a:ABD-Exp}, $u$ is coprime to $k$, and so $u=1$ or $u=q^2$. In the later case, we would have $k=\lambda$, which is a contradiction. Therefore, $u=1$ and $k=\lambda q^2$. Note for $n\geq 4$ that $v-1=s(q)+q^2$, where $s(q)$ is a polynomial divisible by $q^{4}$. Since $k(k-1)=\lambda (v-1)$ and $k=\lambda q^2$, we have that $k =1+[s(q)+q^2]/q^2=[s(q)+2q^2]/q^2$. Therefore, $\lambda=[s(q)+2q^2]/q^4$. Since  $q^4$ divides $s(q)$, it follows that $q^4$ divides $2q^2$, which is impossible. If $i\in\{2,3\}$, by the same argument as in the case where $i=1$, we conclude that $n=5$ and $k=\lambda q^{3}$ if $i=2$, and $n=6$ and $k=\lambda q$ if $i=3$. Thus
    \begin{align*}
    \lambda=\left\{
    \begin{array}{ll}
    \dfrac{q^5+q^2+2}{q^3},& \hbox{if $i=2$;} \\
    \dfrac{q^8+q^7+q^5+q^4+q^3+q^2+2}{q}, & \hbox{if $i=3$.} \\
    \end{array}
    \right.
    \end{align*}
    Since $\lambda$ has to be integer, it follows that $q=2$ when $(n,i)=(6,3)$ in which case $(v,k,\lambda)=(891,446,223)$, but by \cite{a:Braic-2500-nopower}, we have no symmetric design with this parameters set. \smallskip
    
    \noindent \textbf{(2)} Let $H$ be the stabilizer of a nonsingular subspace, and let $q$ be odd. Here by~\cite[Proposition 4.1.4]{b:KL-90}, $H_{0}$ is isomorphic to
    \begin{align*}
    \,^{\hat{}}\SU_{i}(q)\times \SU_{n-i}(q)\cdot (q+1),
    \end{align*}
    where $2i<n$. Then by~\eqref{eq:v} and Lemma \ref{lem:up-lo-b}, we have that $v>q^{2i(n-i)-6}$. It follows from~\cite[p. 336]{a:Saxl2002} that $k$ divides $\lambda d_{ i}(q)$, where $d_{ i}(q)= (q^i-(-1)^i)(q^{n-i}-(-1)^{n-i})$. Then Lemma \ref{lem:six}(b) implies that $\lambda q^{2i(n-i)-2}<\lambda v<k^2\leq \lambda^2(q^i-(-1)^i)^2(q^{n-i}-(-1)^{n-i})^2$. Since $q^i-(-1)^i<2q^i$ and $q^{n-i}-(-1)^{n-i}<2q^{n-i}$, we have that
    \begin{align}\label{eq:eq-A-ii-U}
    q^{2i(n-i)-6}<16\lambda q^{2n}.
    \end{align}
    Since $\lambda$ is an odd prime divisor of $k$, it follows from Lemmas \ref{lem:New} and \ref{lem:six}(b) that $\lambda$ divides $a$, $p$, $q\pm 1$, or $(q^{j}-(-1)^{j})(q-(-1)^{j})^{-1}$ with $j \in \{3,\ldots, n-i\}$. Thus $\lambda \leq \lambda_{i}(q)$, where
    \begin{align}\label{eq:lam-A-ii-U}
    \lambda_{i}(q)=\left\{
    \begin{array}{ll}
    (q^{n-i}+1)(q+1)^{-1},& \text{if} \quad n-i \quad  \text{is odd}; \\
    (q^{n-i-1}+1)(q+1)^{-1}.& \text{if} \quad n-i \quad  \text{is even}.
    \end{array}
    \right.
    \end{align}
    Then by \eqref{eq:eq-A-ii-U} and \eqref{eq:lam-A-ii-U}, we have that $q^{2i(n-i)-6}\cdot (q+1)<32q^{3n-i}$, and so $n(2i-3)+i<2i^2+8$. Since $n>2i$, we conclude that $4i^2-5i<n(2i-3)+i<2i^2+8$, and so $2i^2<5i+8$. This inequality holds only for $i=1, 2, 3$. If $i=3$, then by \eqref{eq:v}, we have that
    \begin{align*}
    v=\frac{q^{3n-9}(q^{n}-(-1)^n)(q^{n-1}-(-1)^{n-1})(q^{n-2}-(-1)^{n-2})}{(q^3+1)(q^2-1)(q+1)}.
    \end{align*}
    Here $k$ divides $\lambda d_{3}(q)$, where $d_{3}(q)=(q^3+1)(q^{n-3}-(-1)^{n-3})$. Thus by Lemma \ref{lem:six}(b), we have that $\lambda v<k^2\leq \lambda^2(q^3+1)^2(q^{n-3}-(-1)^{n-3})^2$. Therefore, $q^{3n-9}(q^{n}-(-1)^n)(q^{n-1}-(-1)^{n-1})(q^{n-2}-(-1)^{n-2})<16\lambda\cdot q^{2n}(q^3+1)(q^2-1)(q+1)$, and hence since $q^{n-i}+1<2q^{n-i}$, it follows from \eqref{eq:lam-A-ii-U} that
    \begin{align*}
    (q^{n}-(-1)^n)(q^{n-1}-(-1)^{n-1})(q^{n-2}-(-1)^{n-2})<32q^{6}(q^3+1)(q^2-1).
    \end{align*}
    Note that $n>2i=6$. Thus $(q^7+1)(q^6-1)(q^5+1)\leq (q^{n}-(-1)^n)(q^{n-1}-(-1)^{n-1})(q^{n-2}-(-1)^{n-2})<32q^{6}(q^3+1)(q^2-1)$, and so $(q^7+1)(q^6-1)(q^5+1)<32q^{6}(q^3+1)(q^2-1)$, which is impossible. If $i=2$, then by \eqref{eq:v}, we have $v=q^{2n-4}(q^{n}-(-1)^{n})(q^{n-1}-(-1)^{n-1})(q^2-1)^{-1}(q+1)^{-1}$. Since $k$ divides $\lambda d_{2}(q)=\lambda (q^2-1)(q^{n-2}-(-1)^{n-2})$, by \eqref{eq:lam-A-ii-U} and Lemma \ref{lem:six}(b), we conclude that $q^{2n-4}(q^{n}-(-1)^{n})(q^{n-1}-(-1)^{n-1})<(q^2-1)^3(q^{n-2}+1)^3$. Since also  $(q^2-1)(q^{n-2}+1)<q^{n-3}(q^{3}-1)$, we have that $q^{2n-4}(q^{n}-(-1)^{n})(q^{n-1}-(-1)^{n-1})<q^{3n-9}(q^{3}-1)^3$. Therefore, $(q^{n}-(-1)^{n})(q^{n-1}-(-1)^{n-1})<q^{n-5}(q^{3}-1)^3$. Note that $q^{n}(q^{n-1}-1)\leq (q^{n}-(-1)^{n})(q^{n-1}-(-1)^{n-1})$. Then $q^{n}(q^{n-1}-1)<q^{n-5}(q^{3}-1)^2$, and so $q^{5}(q^{n-1}-1)<(q^3-1)^3$. Since $n>2i=4$, we conclude that $q^5(q^4-1)\leq q^{5}(q^{n-1}-1)<(q^{3}-1)^2$, and so $q^5(q^4-1)<(q^3-1)^3$, which is impossible. Hence $i=1$, in which case by \eqref{eq:v}, we have that $v=q^{n-1}(q^{n}-(-1)^n)(q+1)^{-1}$. If $n$ is even, then as $q$ is odd, $v$ is even, which is a contradiction. Therefore, $n$ is odd, and hence $v=q^{n-1}(q^{n}+1)(q+1)^{-1}$. Recall that $k$ divides $\lambda d_{1}(q)$, where $d_{1}(q)=(q+1)(q^{n-1}-1)$.  Since $\gcd(v-1, (q+1)(q^{n-1}-1) )=(q^{n-1}-1)(q+1)^{-1}$, it follows that $k$ divides $\lambda f(q)$, where $f(q)=(q^{n-1}-1)(q+1)^{-1}$. Let $u$ be a positive integer such that $uk=\lambda f(q)$. Then by Lemma \ref{lem:six}(a), we have that
    \begin{align}\label{eq:k-lam-A-ii-U}
    k=u\cdot (q^{n}+q+1)+1 \quad \text{and} \quad \lambda=u^2q(q+1)+\frac{u^2(2q+1)+u}{f(q)}.
    \end{align}
    Recall that $uk=\lambda f(q)$, where $f(q)=(q^{n-1}-1)(q+1)^{-1}$. Then by \eqref{eq:lam-A-ii-U} and \eqref{eq:k-lam-A-ii-U}, we have that $u^2(q^n+q+1)+u<(q^{n-2}+1)(q^{n-1}-1)(q+1)^{-2}$. Therefore, $u^2q^n(q+1)^2<(q^{n-2}+1)(q^{n-1}-1)<q^n(q^{n-3}+1)$. Note that $(q^{n-3}+1)(q+1)^{-2}<q^{n-5}$. Thus
    \begin{align}\label{eq:A-ii-1-u-U}
    u^2<q^{n-5}.
    \end{align}
    Since $\lambda$ is integer, by \eqref{eq:k-lam-A-ii-U}, we conclude that $f(q)$  divides $u^2(2q+1)+u$. Thus $q^{n}-1\leq [u^2(2q+1)+u](q+1)$. As $[u^2(2q+1)+u](q+1)<6u^2q^2$, it follows that $q^{n}-1<6u^2q^2$, and so by \eqref{eq:A-ii-1-u-U}, we conclude that $q^{n}-1<6q^{n-3}$. Therefore, $q^{2}\leq 6$, which is impossible as $q$ is odd.\smallskip
    
    \noindent \textbf{(3)} Let $H_{0}$ be the stabilizer of an orthogonal decomposition $V=\oplus V_{j}$ with all $V_{j}$'s isometric, and let $q$ be odd. In this case, by~\cite[Proposition 4.2.9]{b:KL-90}, $H_{0}$ is isomorphic to
    \begin{align*}
    \,^{\hat{}}\SU_{i}(q)^{t}\cdot{(q+1)}^{t-1}\cdot\S_{t}
    \end{align*}
    It follows from \cite[Proposition 3.5]{a:ABCD-PrimeRep} that $v>q^{i^2t(t-1)/2}/(t!)$. Suppose first that $i \geq 2$. Since $\lambda$ is an odd prime divisor of $k$, the parameter $\lambda$ must divide $2a\cdot (t!)\cdot  q^{it(i-1)/2)}(q^{i}-(-1)^i)^t{\cdots}(q^{2}-1)^t(q+1)^{t-1}$. Since $\max\{p, t, (q^{i}-(-1)^i)/2\}<t\cdot (q^{i}-(-1)^i)/2$, it follows that
    \begin{align}\label{eq:lam-A-iii-2-U}
    \lambda <t\cdot (q^{i}-(-1)^i)/2.
    \end{align}
    By \cite[p.336]{a:Saxl2002}, the parameter $k$ divides $\lambda t(t-1)(q^i-(-1)^i)^2$. Then by Lemma \ref{lem:six}(b), we have that $\lambda q^{i^2t(t-1)/2}/(t!)<\lambda v<k^2\leq t^2(t-1)^2\lambda^2(q^i-(-1)^i)^4$, and so
    \begin{align}\label{eq:i-t-2}
    q^{i^2t(t-1)/2}<\lambda\cdot (t!)\cdot  t^2(t-1)^2(q^i-(-1)^i)^4.
    \end{align}
    Then  by \eqref{eq:lam-A-iii-2-U}, we have that $2q^{i^2t(t-1)/2}<(t!)\cdot t^3(t-1)^2(q^i-(-1)^i)^5$. Since $q^i-(-1)^i<2q^i$, it follows that $q^{i^2(t^2-t)/2}<16(t!)\cdot t^3(t-1)^2q^{5i}$. If $t\geq 4$, then by Lemma \ref{lem:factor-bound}(b), we have that $t!<2^{4t(t-3)/3}$, and so $q^{i^2(t^2-t)/2}<2^{[4t(t-3)+12]/3}\cdot t^5q^{5i}$. Note that $t^5\leq 2^{3t}$. Thus, $q^{i^2(t^2-t)/2}<2^{[4t^2-3t+12]/3}\cdot q^{5i}$, and so $q^{3i^2(t^2-t)-30i}<2^{8t^2-6t+24}$. Therefore,
    \begin{align}\label{eq:i-t-2-1-1-1}
    t^2(3i^2-8)<3i^2t+30i-6t+24.
    \end{align}
    Since $t\geq 4$, we have that $t^2(3i^2-8)<3i^2t+30i-6t+24<(3i^2+30i)t$, and so $t(3i^2-8)<3i^2+30i$. Thus $12i^2-32\leq t(3i^2-8)<3i^2+30i$, and so $9i^2-30i<32$. Then $i=2, 3, 4$.
    
    Suppose that $i=2$. Then by \eqref{eq:i-t-2-1-1-1}, we conclude that $4t^2<6t+84$ implying that  $t=4,5$. If $(i, t)=(2, 4)$, then by \eqref{eq:i-t-2} and \eqref{eq:lam-A-iii-2-U}, we have that $q^{24}<4^4\cdot 3^3(q^2-1)^5$, which is impossible. If $(i, t)=(2, 5)$, then by \eqref{eq:i-t-2} and \eqref{eq:lam-A-iii-2-U}, we conclude that $q^{40}<5^4\cdot 4^3\cdot 3(q^2-1)^5$, which is impossible. The case where $i=3, 4$, can be ruled out by the same manner as above.\smallskip
    
    Suppose now that $i=1$. Then $H_{0}$ is isomorphic to $\,^{\hat{}}(q+1)^{n-1}\cdot \Sym_{n}$, and so by~\eqref{eq:v}, we have that
    \begin{align}\label{eq:2-v}
    v=\frac{q^{n(n-1)/2}(q^{n}-(-1)^{n})\cdots(q^2-1)}{(q+1)^{n-1} \cdot n!}.
    \end{align}
    Note that $\lambda$ is an odd prime divisor of $k$. Then $\lambda$ divides $2a(n!)(q+1)^{n-1}$. Therefore, $\lambda$ must divide $a$, $n!$ or $q+1$, and so $\lambda \leq \max\{a, n, (q+1)/2\}$. In conclusion, $\lambda <n(q+1)/2$. We now consider the following subcases:\smallskip
    
    \noindent{\textbf{(3.1)}} Let $q\geq 5$. Here by \cite[p.337]{a:Saxl2002}, we have that $k$ divides $\lambda n(n-1)(q+1)^2/2$. Then Lemma~\ref{lem:six}(b) implies that $
    4q^{n(n-1)/2}(q^{n}-(-1)^{n})\cdots(q^2-1)<\lambda (n!)\cdot n^2(n-1)^2(q+1)^{n+3}$. Recall that $\lambda<n(q+1)/2$. Therefore,
    \begin{align}\label{eq:lam-A-ii-U-3}
    8q^{n(n-1)/2}(q^{n}-(-1)^{n})\cdots(q^2-1)<(n!)\cdot n^3(n-1)^2(q+1)^{n+4}.
    \end{align}
    Note that $q+1<2q$. Then \eqref{eq:lam-A-ii-U-3} and Lemma \ref{lem:equation} imply that $q^{n^2-2n-4}<2^{n+1}(n!)\cdot n^3(n-1)^2<2^{n+1}(n!)\cdot n^5$. As $n\geq 5$, we conclude that $n^5\leq 2^{3n}$. Then by \cite[Lemma 4.4]{a:AB-Large-15}, we have that $q^{3n^2-6n-12}<2^{4n^2+3}$. Since $q\geq 5$, it follows that $2^{6n^2-12n-24}<q^{3n^2-6n-12}<2^{4n^2+3}$. Thus $2n^2<12n+27$,  and so $n=5, 6, 7$. Let now $h_{n}(q)=8q^{(n^{2}-3n-8)/2}(q^{n}-(-1)^{n})\cdots(q^2-1)$. Then since $q+1<2q$, we conclude by \eqref{eq:lam-A-ii-U-3} that $h_{n}(5)\leq h_{n}(q)<2^{n+4}\cdot n^3(n-1)^2(n!)$, for $n\in\{5,6,7\}$. Define
    \begin{align*}
    u_{n}=\left\{
    \begin{array}{ll}
    2^{16}\cdot 3\cdot5^4, & \hbox{if $n=5$;} \\
    2^{17}\cdot 3^{5}\cdot 5^{3}, & \hbox{if $n=6$;} \\
    2^{17}\cdot 3^{4}\cdot 5\cdot 7^{4}, & \hbox{if $n=7$.}
    \end{array}
    \right.
    \end{align*}
    Then
    $h_{n}(q)< u_{n}$ for $n\in\{5,6,7\}$, and hence $h_{n}(5)< u_{n}$, which is impossible.\smallskip
    
    \noindent{\textbf{(3.2)}} Let $q=3$. Here by \cite[p.337]{a:Saxl2002}, we have that $k$ divides $\lambda n(n-1)(n-2)(q+1)^3/6$. Then Lemma~\ref{lem:six}(b) implies that $
    6q^{n(n-1)/2}(q^{n}-(-1)^{n})\cdots(q^2-1)<\lambda (n!)\cdot n^2(n-1)^2(n-2)^2(q+1)^{n+5}$. Recall that $\lambda<n$. Therefore,
    \begin{align}\label{eq:lam-A-ii-U-3-1}
    3q^{n(n-1)/2}(q^{n}-(-1)^{n})\cdots(q^2-1)<2^{2n+9}n^3(n-1)^2(n-2)^2\cdot (n!).
    \end{align}
    Since $q^{n(n-1)/2}\leq (q^{n}-(-1)^{n})\cdots(q^2-1)$, replacing $q$ by $3$, we have that $3^{n^2-n}<2^{2n+9}(n!)\cdot n^3(n-1)^2(n-2)^2<2^{2n+9}(n!)\cdot n^7$. As $n\geq 5$, we conclude that $n^7\leq 2^{4n}$, and so \cite[Lemma 4.4]{a:AB-Large-15} implies that $3^{3n^2-3n}<2^{4n^2+6n+27}$. Therefore, $3n^2-3n<(4n^2+6n+27)\cdot \log_{3}2<(4n^2+6n+27)\times 0.7$, and so $2n^2<72n+189$, and hence $n\in\{5, \ldots, 38\}$. Let
    $h_{n}(q)=3q^{(n^{2}-n)/2}(q^{n}-(-1)^{n})\cdots(q^2-1)$. Then $h_{n}(3)<u_{n}$, where $u_{n}=2^{2n+9}n^3(n-1)^2(n-2)^2\cdot (n!)$. However, it is easy to check that this inequality does not hold for $n\in\{5, \ldots, 38\}$.\smallskip
    
    \noindent \textbf{(4)} Let $H = N_{G}(X(q_{0}))$ with $q=q_{0}^t$ odd and $t$ odd prime. By~\cite[Proposition 4.5.3]{b:KL-90}, the subgroup $H_{0}$ is isomorphic to
    \begin{align*}
    \,^{\hat{}}\SU_{n}(q_{0})\cdot \gcd((q+1)/(q_{0}+1), n).
    \end{align*}
    Since $|\Out(X)|=2a\cdot \gcd(n, q+1)$, by Lemma~\ref{lem:up-lo-b} and the inequality $|X|<|\Out(X)|^2\cdot |H_{0}|^3$, we have that $q_{0}^{t(n^2-2)}<8a^2\cdot  q_{0}^{3n^2}(1+q_{0}^{-1})^3(1+q_{0}^{-3})^3(q_{0}^t+1)^3$. As $a^2<2q$, $q_{0}^t+1<2q_{0}^t$ and $(1+q_{0}^{-1})^3(1+q_{0}^{-3})^3<2$, we have that $q_{0}^{n^2(t-3)-6t}<256$. Note that $q_{0}$ is odd. So $3^{n^2(t-3)-6t}<256$. If $t\geq 5$, then $3^{2n^2-30}<256<3^6$, and so $2n^2-30<6$, which  contradicts the fact that $n\geq 5$. Therefore, $t=3$. In this case, by \eqref{eq:v} and Lemma \ref{lem:up-lo-b}, we have that $v>q_{0}^{2n^2-10}$. By Lemmas \ref{lem:New} and \ref{lem:six}(b), the parameter $k$ divides $2a\cdot q_{0}^{n(n-1)/2}(q_{0}^{n}-(-1)^n){\cdots}(q_{0}^{2}-1)\cdot \gcd(q_{0}^2-q_{0}+1, n)$. It follows from Lemma \ref{lem:six}(a) and (c) that $k$ divides $\lambda \gcd(v-1,|\Out(X)|\cdot |H_{0}|)$.  Since by Lemma~\ref{lem:Tits},  $v-1$ is coprime to $q_{0}$, we conclude that
    \begin{align}\label{eq:k-A-i-U}
    k \text{ divides }  2a\lambda\cdot |H_{0}|_{p'}.
    \end{align}
    Then by \eqref{eq:k-A-i-U} and Lemma \ref{lem:six}(b), we have that $\lambda q_{0}^{2n^2-10}< \lambda v <k^{2}\leq 4a^2\lambda^2 (q_{0}^{n}-(-1)^n)^2{\cdots}(q_{0}^{2}-1)^2\cdot (q_{0}^2-q_{0}+1)^2$. Since $(q_{0}^2-q_{0}+1)^2<q_{0}^4$, we conclude that $q_{0}^{n^2-n-12}<4a^2\lambda$. Since also  $\lambda$ is an odd prime divisor of $|H|$, it must divide $a$, $p$, $q\pm 1$ or $(q_{0}^{j}-(-1)^{j})/(q_{0}-(-1)^{j})$, for some $j\in \{2, 3, \ldots, n\}$. Then $\lambda \leq (q_{0}^{n}-1)/(q_{0}-1)$,  and so the inequality  $q_{0}^{n^2-n-12}<4a^2\lambda$ implies that $q_{0}^{n^2-2n-12}(q_{0}-1)<4a^2$. As $a=3s$ and  $n\geq 5$, it follows that $q_{0}^3(q_{0}-1)\leq q_{0}^{n^2-2n-12}(q_{0}-1)<36s^2$. Therefore, $q_{0}^3(q_{0}-1)<36s^2$, which is impossible.
\end{proof}

\begin{proposition}\label{prop:psp}
    Let $\Dmc$ be a nontrivial symmetric $(v,k,\lambda)$ design with $\lambda$ prime. Suppose that $G$ is an automorphism group of $\Dmc$ of almost simple type with socle $X$. If $G$ is flag-transitive, then the socle $X$ cannot $\PSp_{2m}(q)$ with $(m, q)\neq (2,2), (2,3)$.
\end{proposition}
\begin{proof}
    Let $H_{0}=H\cap X$, where $H=G_{\alpha}$ with $\alpha$ a point of $\Dmc$.  Then by Lemma \ref{lem:six}(a), $v$ is odd, and so by Lemma \ref{lem:odd-deg} one of the following holds:
    \begin{enumerate}[\rm (1)]
        \item $q$ is even, and $H_{0}$ is a parabolic subgroup of $X$;
        \item $q$ is odd, and $H$ is the stabilizer of a nonsingular subspace;
        \item $q$ is odd, and $H_{0}$ is the stabilizer of an orthogonal decomposition $V=\oplus V_{j}$ with all $V_{j}$'s isometric;
        \item $q=q_{0}^t$ is odd with $t$ odd prime, and $H = N_{G}(X(q_{0}))$.
    \end{enumerate}
    We now analyse each of these possible cases separately. \smallskip
    
    \noindent \textbf{(1)} Let $H_{0}$ be a parabolic subgroup of $X$, and let $q=2^{a}$ be even. Then~\cite[Proposition 4.1.19]{b:KL-90} implies that $H_{0}$ is isomorphic to
    \begin{align*}
    [q^{h}]\cdot (\GL_{i}(q)\times \PSp_{2m-2i}(q)),
    \end{align*}
    where $h=2mi+(i-3i^2)/2$ and $i\leq m$. It follows from~\eqref{eq:v} and Lemma \ref{lem:up-lo-b} that  $v>q^{i(4m-3i)}$. By Lemma~\ref{lem:subdeg}, there is a unique subdegree $d=2^{c}$. The $2$-power $(v-1)_{2}$ is $q$. Since $k$ divides $\lambda \gcd(v-1, d )$, it follows that $k$ divides $\lambda q$. By the fact that $\lambda$ is an odd prime divisor of $k$, Lemma \ref{lem:six}(b) implies that $\lambda$ must divide $a$, $q^{j}-1$ with $j\in \{1,\ldots, i\}$ or $q^{2j}-1$ with $j\in\{1,\ldots, m-i\}$. Thus
    \begin{align}\label{eq:lam-B-S}
    \lambda \leq (q^{m}-1)/(q-1).
    \end{align}
    It follows from Lemma \ref{lem:six}(b) that $\lambda q^{i(4m-3i)}<\lambda v<k^2\leq \lambda^2 q^2$, and so $q^{i(4m-3i)}<\lambda q^2$. Then by \eqref{eq:lam-B-S}, we have that
    \begin{align}\label{eq:eq-B-S}
    q^{i(4m-3i)}(q-1)<q^2(q^m-1).
    \end{align}
    Therefore, $i(4m-3i)<m+2$, and so $m(4i-1)<3i^2+2$. Since $i\leq m$, it follows that $i(4i-1)\leq m(4i-1)<3i^2+2$. Thus $i^2<i+2$, and hence $i=1$. By \eqref{eq:eq-B-S}, we have that $q^{4m-3}(q-1)<q^2(q^m-1)$, and so $q^{4m-3}<q^{m+2}$, which is impossible.\smallskip

    \noindent \textbf{(2)} Let $H$ be the stabilizer of a nonsingular subspace, and let $q$ be odd.  Here by~\cite[Proposition 4.1.3]{b:KL-90}, the subgroup $H_{0}$ is isomorphic to
    \begin{align*}
    \PSp_{2i}(q)\times \PSp_{2m-2i}(q) \cdot 2,
    \end{align*}
    where $2i<m$. In this case, $v>q^{4i(m-i)}$, and so Lemma~\ref{lem:six}(c) implies that $k$ divides $\lambda d_{ i}(q)$, where $d_{ i}(q)=(q^{2i}-1)(q^{2m-2i})(q^2-1)^{-2}$. Again by Lemma~\ref{lem:six}(b), we conclude that $\lambda q^{4i(m-i)}< \lambda v < k^{2}\leq \lambda^2 (q^{2i}-1)^2(q^{2m-2i}-1)^2(q^2-1)^{-4}$. Therefore,
    \begin{align}\label{eq:A-ii-1-S}
    q^{4i(m-i)}(q^2-1)^4<\lambda(q^{2i}-1)^2(q^{2m-2i}-1)^2
    \end{align}
    Since $\lambda$ is an odd prime divisor of $k$, Lemmas \ref{lem:New} and \ref{lem:six}(b) imply that $\lambda$ must divide $a$, $p$ or $q^{2j}-1$, for some $j \in \{1, \ldots, m-i\}$, and so
    \begin{align}\label{eq:A-ii-1-lam-S}
    \lambda \leq q^{m-i}+1.
    \end{align}
    Then by \eqref{eq:A-ii-1-S}, we have that $q^{4i(m-i)+6}<q^{5m-i}$. Thus $4i(m-i)+6<5m-i$, and so $m(4i-5)<4i^2-i-6$. As $m>2i$, the last inequality holds only for $i=1$, in which case by \eqref{eq:v}, we have that $v=q^{2m-2}(q^{2m}-1)(q^2-1)^{-1}$ and
    \begin{align}\label{eq:k-A-ii-1-S}
    k \text{ divides }  \lambda d_{1}(q),
    \end{align}
    where $d_{ 1}(q)=(q^{2m-2}-1)(q^2-1)^{-1}$. Let $u$ be a positive integer such that $uk=\lambda d_{1}(q)$. Since $v-1=(q^{2m-2}-1)(q^{2m}+q^2-1)(q^2-1)^{-1}$, by Lemma \ref{lem:six}(a), we have that
    \begin{align}\label{eq:k-lam-A-ii-S}
    k=u\cdot (q^{2m}+q^2-1)+1 \quad \text{and} \quad \lambda=u^2q^2(q^2-1)+\frac{u^2(2q^2-1)+u}{d_{1}(q)}.
    \end{align}
    It follows from  \eqref{eq:k-lam-A-ii-S} and \eqref{eq:k-A-ii-1-S} that $u\cdot (q^{2m}+q^2-1)+1\leq (q^{m-1}+1)(q^{2m-2}-1)(q^2-1)^{-1}$, and so $uq^{2m}(q^2-1)<(q^{m-1}+1)(q^{2m-2}-1)$. Since $(q^{m-1}+1)(q^{2m-2}-1)<q^{2m-2}(q^{m+1}+1)$, we have that $u\leq 2 q^{m-4}$. Since $\lambda$ is a positive integer,  we conclude by \eqref{eq:k-lam-A-ii-S} that $d_{1}(q)$ must divide $u^2(2q^2-1)+u$. Since also $u\leq 2 q^{m-4}$, we have that  $q^{2m-2}-1\leq (u^2(2q^2-1)+u)(q^{2}-1)\leq 2u^2q^2(q^2-1)< 8q^{2(m-4)+2}(q^2-1)$, and so $q^{2m-2}-1< 8q^{2m-6}(q^2-1)$. Thus $q^{2m-2}-1< 8q^{2m-4}$, and hence $q^{2}\leq 8$, which is impossible as $q$ is odd.\smallskip
    
    \noindent \textbf{(3)} Let $H_{0}$ be the stabilizer of an orthogonal decomposition $V=\oplus V_{j}$ with all $V_{j}$'s isometric, and let $q$ be odd. In this case, by~\cite[Proposition 4.2.10]{b:KL-90}, the subgroup $H_{0}$ is isomorphic to $\,^{\hat{}}\Sp_{2i}(q)\wr\Sym_{t}$ with $it=m$. Here by \cite[p.16]{a:ABCD-PrimeRep}, we have that $v>q^{2i^2t(t-1)}/(t!)$. The parameter $\lambda$ divides $k$. Then  it divides $2a\cdot (t!)\cdot q^{i^2t}(q^{2i}-1)^t{\cdots}(q^2-1)^t$ by Lemmas \ref{lem:New} and \ref{lem:six}(b). Therefore, $\lambda$ divides $a$, $p$, $t!$ or $q^{2j}-1$, for some $j\in \{1,\ldots,i\}$, and since $\lambda \leq \max \{a, p, t, (q^{i}+1)/2\}$, it follows that
    \begin{align}\label{eq:lam-A-iii-2-S}
    \lambda < t \cdot (q^{i}+1)/2.
    \end{align}
    Note also by \cite[p.328]{a:Saxl2002} that
    \begin{align}\label{eq:k-A-iii-2-SP}
    k \quad \text{divides} \quad \lambda t(t-1)(q^{2i}-1)^2(q-1)^{-1}/2.
    \end{align}
    Then Lemma \ref{lem:six}(b) implies that $\lambda v<k^2\leq \lambda^2 t^2(t-1)^2(q^{2i}-1)^4(q-1)^{-2}/4$, and so
    \begin{align}\label{eq:A-ii-2-2-eq-S}
    4q^{2i^2t(t-1)}(q-1)^2<\lambda \cdot t^2(t-1)^2(t!)(q^{2i}-1)^4.
    \end{align}
    It follows from \eqref{eq:lam-A-iii-2-S} and \eqref{eq:A-ii-2-2-eq-S} that
    \begin{align}\label{eq:A-ii-eq-S}
    8q^{2i^2t(t-1)}(q-1)^2<t^3(t-1)^2(t!)(q^{2i}-1)^4(q^{i}+1).
    \end{align}
    We now consider the following subcases:\smallskip
    
    \noindent\textbf{(3.1)} Assume first that $t\geq 4$.
    Note by Lemma \ref{lem:factor-bound}(b) that $t!<2^{4t(t-3)/3}$. Then $4^{3}q^{6i^2t(t-1)}{\cdot}(q-1)^6<2^{4t(t-3)}{\cdot}t^{15}(q^{2i}-1)^{12}q^{3i}$, and so $q^{6i^2t^2-6i^2t-27i+3}<2^{4t^2-3t-6}$. Thus $2t^2(3i^2-2)<3t(2i^2-1)+27i-9<3ti(2i+9)$. Therefore,
    \begin{align}\label{eq:A-iii-2-Sp}
    2t(3i^2-2)<3i(2i+9).
    \end{align}
    As $t\geq 4$, it follows that $8(3i^2-2)\leq 2t(3i^2-2)<3i(2i+9)$, and so $8(3i^2-2)<3i(2i+9)$. Then $i=1, 2$. If $i=2$, then by \eqref{eq:A-iii-2-Sp}, we conclude that $20t=2t(3i^2-2)<3i(2i+9)=78$, and so $t<4$, which is a contradiction. Therefore, $i=1$. By \eqref{eq:A-ii-eq-S}, we have that  $8q^{2t(t-1)}(q-1)^2<t^3(t-1)^2(t!)(q^{2}-1)^4(q+1)$. As $t\geq 4$, by Lemma \ref{lem:factor-bound}(b), we conclude that $4q^{2t(t-1)-8}<2^{4t(t-3)/3}{\cdot}t^5$. Since $t^5<2^{3t}$, we have that $q^{6t(t-1)-24}<2^{4t^2-3t-6}$. Thus  $6t(t-1)-24\leq [6t(t-1)-24]{\cdot}\log_{p}q< (4t^2-3t-6)\cdot \log_{p}2<(4t^2-3t-6)\times 0.7$. Therefore, $32t^2-39t<198$, which is impossible for any $t\geq 4$.\smallskip
    
    \noindent\textbf{(3.2)} Assume now that $t=3$. It follows from \eqref{eq:A-ii-eq-S} that $q^{12i^{2}}{\cdot}(q-1)^2<3^4{\cdot}(q^{2i}-1)^4(q^{i}+1)$. Since $q^{2i}-1<q^{2i}$ and $q^{i}+1<2q^{i}$, we  conclude that $q^{12i^{2}}{\cdot}(q-1)^2<3^4{\cdot}(q^{2i}-1)^4(q^{i}+1)<2\cdot 3^{4}q^{9i}$, and so $q^{12i^{2}}<q^{9i+4}$. Thus $i=1$.
    Again, we apply \eqref{eq:A-ii-eq-S} and conclude that $q^{12}{\cdot}(q-1)^2<3^4{\cdot}(q^{2}-1)^4(q+1)$, which is impossible.
    
    \noindent\textbf{(3.3)} Assume finally that $t=2$. The inequality \eqref{eq:A-ii-eq-S} implies that $q^{4i^2}{\cdot}(q-1)^2<2(q^{2i}-1)^4(q^i+1)$. Since $q^i+1<2q^{i}$, we have that $q^{4i^2}{\cdot}(q-1)^2<2q^{9i}$. This inequality holds only for $i\in \{1,2\}$. If $i=2$, then $m=4$, and so by \eqref{eq:v}, we have that $v=q^8(q^4+q^2+1)(q^4+1)/2$. By \eqref{eq:k-A-iii-2-SP} and Lemma \ref{lem:six}(a), the parameter $k$ must divide $\lambda (q^{2}+1)^2$. Then by Lemma \ref{lem:six}(b), we conclude that $\lambda q^8(q^4+q^2+1)(q^4+1)/2<k^2\leq \lambda^2 (q^{2}+1)^4$, and so $q^8(q^4+q^2+1)(q^4+1)< 2\lambda (q^{2}+1)^4$. Then \eqref{eq:lam-A-iii-2-S} implies that $q^8(q^4+q^2+1)(q^4+1)< (q^{2}+1)^5$, which is impossible. Therefore, $i=1$, and hence \eqref{eq:v} implies that
    \begin{align}\label{eq:v-A-ii-2-2-2-SP}
    v=\frac{q^2(q^2+1)}{2}.
    \end{align}
    Since $\gcd(v-1,q+1)$ divides $\gcd(3,q+1)$, it follows from   \eqref{eq:k-A-iii-2-SP} and Lemma~\ref{lem:six}(a) that
    \begin{align*}
    k \text{ divides } c_{1}\lambda f(q),
    \end{align*}
    $f(q)=q^2-1$ and $c_{1}=\gcd(3, q+1)$. Let now $u$ be a positive integer such that $uk=c_{1}\lambda f(q)$. Then Lemma \ref{lem:six} implies that
    \begin{align}\label{eq:k-lam-A-ii-2-2-2-SP}
    2c_{1}k=u\cdot (q^2+2)+2c_{1} \quad \text{and} \quad  2c_{1}^{2}\lambda=u^2+\frac{3u^2+2c_{1}u}{q^2-1}.
    \end{align}
    Recall that $k$ divides $\lambda\cdot c_{1} f(q)$. Then \eqref{eq:k-lam-A-ii-2-2-2-SP} implies that $u(q^2+2)+2c_{1}\leq 2\lambda \cdot c_{1}^2 f(q)$, and so
    \begin{align}\label{eq:u-A-ii-2-SP}
    u< 2 c_{1}^{2}\lambda.
    \end{align}
    Note that $\lambda$ is an odd prime divisor of $k$. Then Lemmas \ref{lem:New} and \ref{lem:six}(b) implies that $\lambda$ divides $4aq^{2}(q^{2}-1)^2$. Therefore, $\lambda$ divides $a$, $p$, $(q-1)/2$ or $(q+1)/2$. We now analyse each of these possibilities.\smallskip
    
    \noindent\textbf{(3.3.1)} Let $\lambda$ divides $a$. By \eqref{eq:u-A-ii-2-SP}, we have that $u<2a\cdot c_{1}^{2}$. Note that $\lambda$ is an integer number. Then \eqref{eq:k-lam-A-ii-2-2-2-SP} implies that $q^2-1$ must divide $3u^2+2c_{1}u$, where $u<2a\cdot c_{1}^{2}$.  Thus $q^2-1\leq 3u^2+2c_{1}u\leq 12a^2\cdot c_{1}^{4}+4a\cdot c_{1}^3$, and so $q^2<16a^2\cdot c_{1}^{4}$, where $c_{1}=\gcd(3, q+1)$, and this holds only for the pairs $(p,a)\in\{(3,1), (5,1)\}$, and so $\lambda$ divides $a=1$, which is impossible.\smallskip
    
    \noindent\textbf{(3.3.2)} Let $\lambda$ divides $p$. Since $\lambda>1$, we have that $\lambda=p$, and so by \eqref{eq:u-A-ii-2-SP}, we have that $u<2p\cdot c_{1}^{2}$. As $\lambda$ is  a positive integer, it follows from \eqref{eq:k-lam-A-ii-2-2-2-SP} that $q^2-1$  divides $3u^2+2c_{1}u$, where $u<2p\cdot c_{1}^{2}$. Thus $q^2-1\leq 3u^2+2c_{1}u\leq 12p^2\cdot c_{1}^{4}+4p\cdot c_{1}^3$, and so $q^2<16p^2\cdot c_{1}^{4}$, where $c_{1}=\gcd(3, q+1)$. Thus either $(p,a)=(3,2)$, or $a=1$. If $(p,a)=(3,2)$, then by \eqref{eq:v-A-ii-2-2-2-SP}, we have that $v=3321$ and that  $k$ divides $\lambda\cdot c_{1} f(p^{2})=3f(9)=240$. Since $\lambda=3$, we conclude that $3(v-1)=k(k-1)$, for some divisor $k$ of $240$, which is a contradiction. Thus $a=1$, and so $q=p$. Since $\lambda$ is an odd prime divisor of $k$, it follows from  \eqref{eq:k-lam-A-ii-2-2-2-SP} that $q=p$ must divide $u+c_{1}$. Let now $u_{1}$ be a positive integer such that $u=u_{1}p-c_{1}$. Then by \eqref{eq:u-A-ii-2-SP}, $u_{1}p-c_{1}<2p\cdot c_{1}^{2}$, and since $p\geq 3\geq c_{1}=\gcd(3,p+1)$, we have that  $u_{1}p<2p\cdot c_{1}^{2}+c_{1}\leq 2p\cdot c_{1}^{2}+p$, and so
    \begin{align}\label{eq:u1-A-ii-2-SP}
    u_{1}< 2c_{1}^{2}+1.
    \end{align}
    Note by \eqref{eq:k-lam-A-ii-2-2-2-SP} that $p^2-1$ divides $3u^2+2u\cdot c_{1}$, where $u=u_{1}p-c_{1}$. Thus $p^2-1$ divides $3u_{1}^2p^2-4u_{1}p\cdot c_{1}+c_{1}^2$. Since $3u_{1}^2p^2-4u_{1}p\cdot c_{1}+c_{1}^2=3u_{1}^{2}(p^{2}-1)+3u_{1}^2-4u_{1}p\cdot c_{1}+c_{1}^2$, we conclude that $p^2-1$ divides $|3u_{1}^2-4u_{1}p\cdot c_{1}+c_{1}^2|$. Set $f_{p}(x):=3x^2-4 pc_{1}x+c_{1}^2$. For a fixed $p$, the quadratic equation $f_{p}(x)=0$ has roots $x_{i} =c_{1}\cdot[2p+(-1)^{i}\sqrt{4p^2-3}]/3$, for $i=1,2$. Note that $x_{2}>2c_{1}^{2}$ and $x_{1}<1$. Therefore, for all $u_{1}$ satisfying \eqref{eq:u1-A-ii-2-SP}, we have that $f_{p}(u_{1})<0$. Thus, $p^2-1$ must divide $-f_{p}(u_{1})=-3u_{1}^2+4u_{1}p\cdot c_{1}-c_{1}^2$. Then $p^2\leq 4 pc_{1}u_{1}$, and so by \eqref{eq:u1-A-ii-2-SP}, we conclude that $p\leq 8c_{1}^3$. This inequality holds only for $p=5,7, \ldots, 197$, for these values of $p$, as $c_{1}=\gcd(3, p+1)$, we can find $u_{1}$ by \eqref{eq:u1-A-ii-2-SP}, but for the pairs $(p,u_{1})$, we observe that $p^2-1$ does not divide $3u_{1}^2p^2-4u_{1}p\cdot c_{1}+c_{1}^2$ except for the case where $(p, u_{1}) \in\{(5,1),(5,3),(5,5),(11,1),(11,3),(17,3)\}$. Note here that $q=p$. Then for each such pairs $(p, u_{1})$, by \eqref{eq:v-A-ii-2-2-2-SP}, we can obtain $v$ as in the second column of Table~\ref{tbl:PSp-A-ii}. Recall that $k$ divides $\lambda\cdot c_{1} f(p)$, where $f(p)=p^2-1$, $\lambda =p$ and $c_{1}=\gcd(3, p+1)$, and so we can find the possible values of $k$ as in the third column of Table~\ref{tbl:PSp-A-ii}.  This is a contradiction as for each $k$ and $v$ as in Table~\ref{tbl:PSp-A-ii}, the equality $p=\lambda=k(k-1)/(v-1)$ does not hold. \smallskip
    
    \begin{table}
        \small
        \centering
        \caption{Some parameters for Case 3.3.2 in Proposition~\ref{prop:psp}}\label{tbl:PSp-A-ii}
        \begin{tabular}{clll}
            \noalign{\smallskip}\hline\noalign{\smallskip}
            Line &
            $(p, a)$ &
            $v$ &
            $k$ divides
            \\
            \noalign{\smallskip}\hline\noalign{\smallskip}
            $1$ &
            $(5,1)$  &
            $325$ &
            $360$
            \\
            $2$ &
            $(11,1)$ &
            $7381$ &
            $ 3960$
            \\
            $3$ &
            $(17,3)$ &
            $41905$ &
            $14688$
            \\
            %
            %
            \noalign{\smallskip}\hline\noalign{\smallskip}
        \end{tabular}
    \end{table}

    \noindent\textbf{(3.3.3)} Let $\lambda$ divides $(q-\e1)/2$, where $\e\in\{+,-\}$. Then $\gcd(\lambda, p)=1$. On the other hand by Lemma \ref{lem:Tits}, we know that $\gcd(v-1, p)=1$, and so Lemma \ref{lem:six}(a) implies that $\gcd(k,p)=1$. It follows from Lemmas~\ref{lem:New} and~\ref{lem:six}(a) that $k$ divides $|\Out(X)|\cdot |H\cap X|_{p'}=4a(q^{2}-1)^{2}$. Then \eqref{eq:k-lam-A-ii-2-2-2-SP} implies that
    \begin{align}\label{eq:G2-SL-k-u}
    u\cdot (q^2+2)+2c_{1} \text{ divides } 8a\cdot c_{1}(q^2-1)^2.
    \end{align}
    Note that $8ac_{1}u\cdot (q^2-1)^2=8ac_{1}h(q)[u\cdot (q^2+2)+2 c_{1}]+G(u,q)$, where $h(q) =q^2-4$ and $G(u,q)=8ac_{1}[9u-2c_{1}h(q)]$. Then $G(u,q)=0$ or
    we conclude by \eqref{eq:G2-SL-k-u} that
    \begin{align}\label{eq:G2-SL-k-m-2}
    u\cdot (q^2+2)+2\cdot c_{1} \text{ divides }  |G(u,q)|.
    \end{align}
    
    Suppose that $G(u,q)=0$. Then $9u=2c_{1}h(q)$. Then $u=2c_{1}h(q)/9=2c_{1}(q^2-4)/9$. Then \eqref{eq:k-lam-A-ii-2-2-2-SP} implies that $\lambda=2(q^2-1)(q^{2}-4)/81$, which is impossible.
    
    Suppose now that $G(u,q)>0$. Then $u>2(q^2-4)/9$ and by \eqref{eq:G2-SL-k-m-2}, $u\cdot (q^2+2)+2\cdot c_{1}<|G(u,q)|=72ac_{1}u-16ac_{1}^{2}h(q)\leq  72ac_{1}u$, and so $q^2+2<72a c_{1}$. Since $r(q) = 9$, it follows that $q^2+2<72a\cdot c_{1}$. This inequality holds when $q=3,5,7,9,11$. Note by \eqref{eq:u-A-ii-2-SP} that $u<c_{1}^{2}(q+1)$ as $\lambda$ divides $(q-\e1)/2$. Thus for $q\in\{3,5,7,9,11\}$, as $2(q^2-4)/9 <u<c_{1}^{2}(q+1)$, we have that $(q,u)\in\{(3,2), (3,3), (5,5), (5,6), \ldots,(5,17), (11,27), (11,28), \ldots,(11,35)\}$. We can now check \eqref{eq:k-lam-A-ii-2-2-2-SP} for these pairs $(q,u)$, and observe that for no such pairs, $\lambda$ is prime.
    
    Suppose finally that $G(u,q)<0$. Then \eqref{eq:G2-SL-k-m-2} implies that $u\cdot (q^2+2)+2c_{1}<|G(u,q)|= 16ac_{1}^{2}h(q)-72ac_{1}u<16ac_{1}h(q)=16ac_{1}\cdot (q^{2}-4)$, and so $u< 16ac_{1}$. Note by \eqref{eq:k-lam-A-ii-2-2-2-SP} that $q^{2}-1$ divides $3u^{2}+2 c_{1}u$. Then $q^{2}-1\leq 3u^{2}+2 c_{1}u<3\cdot 16^{2}a^{2}c_{1}^{2}+2\cdot 16ac_{1}^{2}$, and so $q^{2}-1<2^{10}a^{2}c_{1}^{2}$, and this holds only for $q=p^{a}$ as in Table \ref{tbl:G2-SL}. For each $q$, we can find an upper bound for $u$ listed in the same table, and it is easy to check by \eqref{eq:k-lam-A-ii-2-2-2-SP} that these possible pairs $(q,u)$ give rise to no possible parameters with $\lambda$ prime. \smallskip
    
    \begin{table}
        \small
        \centering
        \caption{Some parameters for Case 3.3.3 in Proposition~\ref{prop:psp}}\label{tbl:G2-SL}
        \begin{tabular}{cccccc} \noalign{\smallskip}\hline\noalign{\smallskip}
            $p$ &
            $3$ &
            $5$ &
            $7$ &
            $11$, $13$, \ldots $89$  \\
            \noalign{\smallskip}\hline\noalign{\smallskip}
            $a\leq $ &
            $4$ &
            $3$ &
            $2$ &
            $1$ \\
            $u< $ &
            $82$ &
            $378$ &
            $50$ &
            $270$ \\
            \noalign{\smallskip}\hline\noalign{\smallskip}
        \end{tabular}
    \end{table}

    \noindent \textbf{(4)}  Let $H = N_{G}(X(q_{0}))$ with $q=q_{0}^t$ odd and $t$ odd prime. Then by~\cite[Proposition 4.5.4]{b:KL-90}, the subgroup $H_{0}$ is isomorphic to $\PSp_{2m}(q_{0})$ with $q=q_{0}^t$. As $|\Out(X)|$ divides $2a$,  by Lemma~\ref{lem:up-lo-b} and Corollary \ref{cor:large}, we have that $q_{0}^{tm(2m-1)}<16a^2{\cdot}q_{0}^{3m(2m+1)}$. Since $a^2<2q$,  it follows that
    \begin{align}\label{eq:A-i-U}
    q_{0}^{t(2m^2-m-1)}<32{\cdot}q_{0}^{6m^2+3m}.
    \end{align}
    As $q_{0}$ is odd, $q_{0}^{t(2m^2-m-1)}<q_{0}^{6m^2+3m+4}$. Thus  $t(2m^2-m-1)<6m^2+3m+4$. If $t\geq 9$, then $9(2m^2-m-1)\leq t(2m^2-m-1) <6m^2+3m+4$, and so $12m^2<12m+13$, which is impossible. Therefore, $t=3, 5, 7$. If $t=7$, then by \eqref{eq:A-i-U}, we have that $q_{0}^{8m^2-10m-7}<32$. As $m\geq 2$ and $q_{0}$ is odd, $3^{5}\leq q_{0}^{8m^2-10m-7}<32$, and so $3^{5}<32$, which is impossible. If $t=5$, then \eqref{eq:A-i-U} implies that $q_{0}^{4m^2-8m-5}<32$, and this inequality holds only for $m=2$. If $(m, t)=(2, 5)$, then by \eqref{eq:v}, we have that
    \begin{align*}
    v=\frac{q_{0}^{16}(q_{0}^{20}-1)(q_{0}^{10}-1)}{(q_{0}^{4}-1)(q_{0}^{2}-1)}>q_{0}^{35}.
    \end{align*}
    By Lemmas \ref{lem:New} and \ref{lem:six}(b), the parameter $k$ divides $2a{\cdot}q_{0}^{4}(q_{0}^{4}-1)(q_{0}^{2}-1)$. It follows from Lemmas \ref{lem:six} and~\ref{lem:Tits} that $k$ divides $2\lambda a{\cdot}(q_{0}^{4}-1)(q_{0}^{2}-1)$. Then by Lemma~\ref{lem:six}(b), we conclude that $\lambda q_{0}^{35}< \lambda v <k^{2}\leq 4\lambda^2 a^2{\cdot}(q_{0}^{4}-1)^2(q_{0}^{2}-1)^2<4\lambda^2 a^2{\cdot}q_{0}^{12}$. Hence, $q_{0}^{23}< 4\lambda a^2$. Since $k$ divides $2a{\cdot}q_{0}^{4}(q_{0}^{4}-1)(q_{0}^{2}-1)$ and $\lambda$ is an odd prime divisor of $k$, we conclude that $\lambda \leq q_{0}^2+1$. Then the inequality $q_{0}^{23}< 4a^2\lambda$ implies that $q_{0}^{21}<8a^2$, and since $a=ts=5s$, it follows that $q_{0}^{21}<200{\cdot} s^2$, which is impossible. Hence $t=3$. In this case by \eqref{eq:v} and Lemma \ref{lem:up-lo-b}, we have that $v>q_{0}^{4m^2-4m-2}$. It follows from Lemmas \ref{lem:New} and \ref{lem:six} and Tits' Lemma~\ref{lem:Tits} that $k$ divides $2a\lambda{\cdot}g(q_{0})$, where $g(q_{0})=(q_{0}^{2m}-1){\cdots}(q_{0}^{2}-1)$. By Lemma \ref{lem:six}(b), we conclude that $\lambda q_{0}^{4m^2-4m-2}< \lambda v <k^{2}\leq 4\lambda^2 a^2{\cdot}(q_{0}^{2m}-1)^2{\cdots}(q_{0}^{2}-1)^2$. Thus
    \begin{align}\label{eq:eq-A-i-S}
    q_{0}^{2m^2-6m-2}<4a^2\lambda.
    \end{align}
    Note that $\lambda$ is an odd prime divisor of $k$ and $k$ divides $|H|$. Then $\lambda$ must divide $a$, $p$ or $(q_{0}^{2j}-1)$, for some $j\in\{1, \ldots, m\}$, and so
    \begin{align}\label{eq:lam-A-i-S}
    \lambda \leq q_{0}^m+1.
    \end{align}
    Therefore, by the inequality \eqref{eq:eq-A-i-S}, we have that $q_{0}^{2m^2-6m-2}<4a^2{\cdot}(q_{0}^m+1)$. As $a=ts=3s$ and $q_{0}^m+1<2q_{0}^m$, we conclude that $q_{0}^{2m^2-7m-2}<72s^2$ implying that  $m=2, 3, 4$. If $m=2$, then by \eqref{eq:v}, we have that $v=q_{0}^{8}(q_{0}^{8}+q_{0}^{4}+1)(q_{0}^{4}+q_{0}^{2}+1)>q_{0}^{20}$. Here by Lemma \ref{lem:six}(a)-(c), $k$ divides $2\lambda a {\cdot}\gcd(v-1, |H\cap X|)$. Then by Lemma \ref{lem:Tits} and the fact that $\gcd(v-1, q_{0}^2+1)=2$, we conclude that $k\leq 4\lambda a {\cdot}(q_{0}^2-1)^2$. So Lemma \ref{lem:six}(b) implies that $\lambda q_{0}^{20}<\lambda v< k^2\leq 16\lambda^2 a^2(q_{0}^2-1)^4$. Thus, $q_{0}^{12}<16\lambda a^2$. So by \eqref{eq:lam-A-i-S}, we have that $q_{0}^{12}<16a^2{\cdot}(q_{0}^2+1)$. Recall that $a=3s$ and $q_{0}^2+1<2q_{0}^2$. Then $q_{0}^{12}<2^5{\cdot}3^2{\cdot}s^2$, which is impossible. By the same manner as above, the remaining cases where $m=3, 4$ can be ruled out.
\end{proof}

\begin{proposition}\label{prop:orth}
    Let $\Dmc$ be a nontrivial symmetric $(v,k,\lambda)$ design with $\lambda$ prime. Suppose that $G$ is an automorphism group of $\Dmc$ of almost simple type with socle $X$. If $G$ is flag-transitive, then the socle $X$ cannot be $\POm_{n}^{\e}(q)$ with $\e \in \{\circ, -, +\}$.
\end{proposition}
\begin{proof}
    Let $H_{0}=H\cap X$, where $H=G_{\alpha}$ with $\alpha$ a point of $\Dmc$. Note by Lemma \ref{lem:six}(a) that $v$ is odd, and so by Lemma \ref{lem:odd-deg}, we have one of the following possibilities:
    \begin{enumerate}[\rm (1)]
        \item $q$ is even, and $H_{0}$ is a parabolic subgroup of $X$;
        \item $q$ is odd, and $H$ is the stabilizer of a nonsingular subspace;
        \item $q$ is odd, and $H_{0}$ is the stabilizer of an orthogonal decomposition $V=\oplus V_{j}$ with all $V_{j}$'s isometric;
        \item $H_{0}$ is $\SO_{7}(2)$ or $\Omega_{8}^{+}(2)$ and $X$ is $\Omega_{7}(q)$ or $\POm_{8}^{+}(q)$, respectively, $q=p\equiv \pm 3 \mod{8}$;
        \item $X=\POm_{8}^{+}(q)$, $q=p\equiv \pm 3 \mod{8}$, $G$ contains a triality automorphism of $X$ and $H_{0}$ is $2^3\cdot 2^6\cdot \PSL_{3}(2)$;
        \item $q=q_{0}^t$ is odd with $t$ odd prime, and $H = N_{G}(X(q_{0}))$.
    \end{enumerate}
    Note in the cases (1) and (6) for $X=\Omega_{2m+1}(q)$ that we argue exactly the same as in the symplectic groups. Therefore, we exclude these possibilities, and analyse the remaining cases.\smallskip
    
    \noindent \textbf{(1)} Let  $H_{0}$ be a parabolic subgroup of $X$, and let $q$ be even. As noted above, we only need to consider the case where $X=\POm_{2m}^{\e}(q)$ with $(m,\e)\neq (2,+)$, $\e=\pm$ and $q$ even. We postpone the case where $(m,\e)=(4,+)$ and $G$ contains a triality automorphism till the end of this case. In this case by~\cite[Proposition 4.1.20]{b:KL-90}, $H_{0}$ is isomorphic to $[q^{h}]\cdot \GL_{i}(q)\times \Omega_{2m-2i}^{\e}(q)$, where $h=2mi-(3i^2+i)/2$.
    
    Suppose first that $H$ stabilises a totally singular $i$-space with $i\leq m-1$, and so $H=P_{i}$ excluding the case where $i=m-1$ and $\e=+$, where  $H=P_{m, m-1}$. It follows from~\eqref{eq:v} and Lemma \ref{lem:up-lo-b} that $v>2^{-5}q^{(4mi-3i^2-i-2)/2}$.  Note that $\lambda$ is an odd prime divisor of $k$ and $|\Out(X)|$ divides $6a$. Then by Lemma \ref{lem:six}(b), we have that
    \begin{align}\label{eq:lam-B-O+-}
    \lambda \leq q^{m-1}+1.
    \end{align}
    In all cases, by Lemma~\ref{lem:subdeg} there is a unique subdegree $d$ of $X$ that is a power of $p$ except for the case where $\e=+$, $m$ is odd and $H=P_m$ or $P_{m-1}$. Note that the $p$-part $(v-1)_{p}$ of $v-1$ is $q^2$ or $8$. Since $k$  divides $\lambda \gcd(v-1, d)$, it follows that $k$ divides $\lambda q^3$. It follows from Lemma \ref{lem:six}(b) that $\lambda q^{(4mi-3i^2-i-2)/2}<32\lambda v < 32 k^2\leq 32\lambda^2 q^{6}$. Therefore, $q^{(4mi-3i^2-i-2)/2}<32\lambda q^{6}$, an so by \eqref{eq:lam-B-O+-}, we have that
    \begin{align}\label{eq:eq-B-O+-}
    q^{(4mi-3i^2-i-2)/2}<32q^{6}(q^{m-1}+1).
    \end{align}
    Then $q^{(4mi-3i^2-i-2)/2}<2^5\cdot q^6(q^{m-1}+1)$. Since $q^{m-1}+1<2q^{m-1}$, we have that $q^{2m(2i-1)-3i^2-i-12}<2^{12}$. Since also $m\geq i+1$, it follows that $2(i+1)(2i-1)\leq 2m(2i-1)<3i^2+i+24$, and so $i^2+i<26$, then $i\in\{1, 2, 3, 4\}$. If $i=1$, then $2m=2m(2i-1)<3i^2+i+24=28$, and so $m=4, \ldots, 13$. By \eqref{eq:v}, we have that $v=(q^m-\e1)(q^{m-1}+\e1)/(q-1)$. Recall that there is a unique subdegree $d$ of $X$ that is a power of $p$. Since $k$ divide $\lambda \gcd(v-1, d)$, it follows that $k$ divides $\lambda q$. Thus Lemma \ref{lem:six}(b) that $\lambda(q^{m-1}+\e1)(q^m-\e1)/(q-1)\leq \lambda v<k^2\leq \lambda^2 q^2$, and so $(q^{m-1}+\e1)(q^m-\e1)<\lambda q^2(q-1)$. Then by \eqref{eq:lam-B-O+-}, we have that $(q^m-\e1)(q^{m-1}+\e1)<q^2(q-1)(q^{m-1}+1)$. If $\e=+$, then $(q^{m-1}+1)(q^m-1)<q^2(q-1)(q^{m-1}+1)$, and so $(q^m-1)<q^2(q-1)$, which does not hold for any $m\geq 4$, which is a contradiction. If $\e=-$, then $(q^{m-1}-1)(q^m+1)<q^2(q-1)(q^{m-1}+1)$, and so $q^{2m-1}-q^m+q^{m-1}-1<q^{m+2}-q^{m+1}+q^3-q^2$, and so $q^{2m-3}-q^{m-2}+q^{m-3}\leq q^{m}-q^{m-1}+q-1$. Since $q^{m-3}>q-1$, $q^{m-2}(q^{m-1}-1) \leq q^{m-2}(q^{2}-q)$, and so $(q^{m-1}-1) \leq (q^{2}-q)$, which is impossible.  For the remaining cases $i=2, 3, 4$, we argue exactly as in the case where $i=1$.\smallskip
    
    Suppose finally that $H=P_{m}$ when $X=\POm_{2m}^{+}(q)$. Note that here $P_{m-1}$ and $P_{m}$ are the stabilizers of totally singular $m$-spaces from the two different $X$-orbits. Here by  \eqref{eq:v}, we have that
    \begin{align}
    v=(q^{m-1}+1)(q^{m-2}+1)\cdots(q+1)>q^{m(m-1)/2}.
    \end{align}
    Note that $\lambda$ is an odd prime divisor of $k$ and $|\Out(X)|$ divides $6a$. Then by Lemma \ref{lem:six}(b), $\lambda$ must divide $3$, $a$ or $q^{j}-1$, for some $j\in \{1,\ldots, m\}$. Thus
    \begin{align}\label{eq:lam-B-2-O+-}
    \lambda \leq (q^{m}-1)/(q-1).
    \end{align}
    
    Assume that $m$ is even. Note by \cite[p. 332]{a:Saxl2002} that there is a subdegree $d$  which is a power of $p$. On the other hand, the $p$-part of $v-1$ is $q$. Since $k$ divides $\lambda \gcd(v-1, d)$, we have that $k$ divides $\lambda q$, and so Lemma \ref{lem:six}(b) implies that $\lambda q^{m(m-1)/2}<\lambda v<k^2<\lambda^2 q^{2}$, and so $q^{m(m-1)/2}<\lambda q^2$. Thus by \eqref{eq:lam-B-2-O+-}, we conclude that $q^{m(m-1)/2}(q-1)<(q^m-1)q^2$, and so $m(m-1)<2m+4$, which is impossible for $m\geq 4$. \smallskip
    
    Assume that $m$ is odd. Then \cite[p. 332]{a:Saxl2002} implies that $k$ divides $\lambda q(q^m-1)$, and so by Lemma \ref{lem:six}(b), we have that $\lambda q^{m(m-1)/2}<\lambda v<k^2<\lambda^2 q^{2}(q^m-1)^2$. Thus $q^{m(m-1)/2}<\lambda q^{2}(q^m-1)^2$. Then \eqref{eq:lam-B-2-O+-} implies that $q^{m(m-1)/2}(q-1)<q^{2}(q^m-1)^3$, and so $m(m-1)<6m+4$, then $m=5, 7$. If $m=5$, then action here is of rank three. The symmetric designs with a primitive rank $3$ automorphism group have been classified by Dempwolff  \cite{a:Dempwolff2001}, we know that there is no such symmetric design with $\lambda$ prime. If $m=7$, then since $k$ divides $\lambda q(q^7-1)$ and $\gcd(v-1, q^6+q^5+q^4+q^3+q^2+q+1)=1$, the parameter $k$ must divide $\lambda q(q-1)$. It follows from Lemma \ref{lem:six}(b), that $\lambda q^{21}<\lambda v<k^2<\lambda^2 q^{2}(q-1)^{2}$, and so $q^{21}<\lambda q^2(q-1)^{2}$. Thus by \eqref{eq:lam-B-2-O+-}, we conclude that $q^{21}<q^2(q-1)^{2}(q^7-1)$, which is impossible.\smallskip
    
    Let now $X=\POm_{8}^{+}(q)$, and let $G$ contain a triality automorphism. We use \cite[Table 8.50]{b:BHR-Max-Low}, where the maximal subgroups are determined. By case (1),  we only need to consider the case where $H\cap X$ is isomorphic to $[q^{11}]{:}(q-1)^2\cdot \GL_{2}(q)$. By \eqref{eq:v}, we have that $v=(q^6-1)(q^4-1)^2/(q-1)^3> q^{11}$.  Since the $p$-part of $v-1$ is $q$ and $k$ divides $\lambda \gcd(v-1, d)$, it follows that $k$ divides $\lambda q$. Then Lemma \ref{lem:six}(b) implies that $\lambda q^{11}<\lambda v<k^2<\lambda^2 q^{2}$, and so $q^{11}<\lambda q^2$. Note that $\lambda$ is and odd prime divisor of $k$ dividing $|\Out(X)|\cdot |H\cap X|$. Then Note that $\lambda$ is an odd prime divisor of $k$ and $|\Out(X)|$ divides $6a$. Then $\lambda$ must divide $3$, $a$ or $q^j-1$, for some $j\in\{1,2\}$, and so $
    \lambda \leq \max\{3, q+1\}\leq q+1$. Recall that $q^{11}<\lambda q^2$. Therefore,  $q^{11}<q^2(q+1)$, which is a impossible.\smallskip

    \noindent \textbf{(2)} Let $H$ be the stabilizer of a nonsingular subspace, and let $q$ be odd. Here, we need to discuss the odd and even dimension of the underlying orthogonal space separately.   \smallskip
    
    \noindent \textbf{(2.1)} Let $X=\Omega_{2m+1}(q)$ with $q$ odd and $m\geq 3$. In this case $H=N_{i}^{\e}$ with $i\leq m$. If $i=1$, then by~\cite[Proposition 4.1.6]{b:KL-90}, $H_{0}$ is isomorphic to $\,^{\hat{}}\Omega_{2m}^{\e}(q){\cdot}2$ with $\e\in\{+, -\}$. It follows from \eqref{eq:v} that $v=q^{m}(q^{m}+\e1)/2$. Note here that if $\e=-$, then $m$ is odd as $v$ must be odd.
    
    According to \cite[p.331-332]{a:Saxl2002}, $k$ must divide $\lambda d^{\e}(q)$, where $d^{\e}(q)=(q^m-\e1)/2$. Let $u$ be a positive integer such that $uk=\lambda (q^m-\e1)/2$. Then by Lemma \ref{lem:six}(a), we have that
    \begin{align}\label{eq:k-lam-A-ii-Oo}
    k=u\cdot (q^{m}+\e2)+1 \quad \text{and} \quad \lambda=2u^2+\frac{\e 3u^2+u}{d^{\e}(q)}.
    \end{align}
    Since $uk=\lambda d^{\e}(q)=\lambda (q^m-\e1)/2$, it follows from Lemma \ref{lem:six}(b) that $\lambda v<k^2\leq \lambda^2 d^{\e}(q)^2/u^2$. Therefore,
    \begin{align}\label{eq:eq-A-ii-1-Oo}
    2u^2q^m(q^m+\e1)<\lambda (q^m-\e1)^2.
    \end{align}
    Note that here $|\Out(X)|=2a$ and $\lambda$ is an odd prime divisor of $k$. Then by Lemmas \ref{lem:New} and \ref{lem:six}(b), $\lambda$ must divide $a$, $p$, $q^{m}-\e1$ or $q^{2j}-1$, for some $j\in\{1,\ldots, m-1\}$. Considering all these possible cases, it is easy to see that $\lambda \leq (q^{m}-\e1)/(q-\e1)$. So by \eqref{eq:eq-A-ii-1-Oo}, we have that $2u^2q^m(q^m+\e1)<\lambda (q^m-\e1)^2/(q-\e1)\leq (q^m-\e1)^3/(q-\e1)$.  Therefore,
    \begin{align}\label{lem31-O-u}
    u^2<\frac{(q^m-\e1)^3}{2q^m(q^m+\e1)(q-\e1)}.
    \end{align}
    Note that $\lambda$ is an integer number. Then  \eqref{eq:k-lam-A-ii-Oo} implies that $d^{\e}(q)$ must divide $|\e 3u^2+u|$, where $d^{\e}(q)=(q^m-\e1)/2$.
    
    Let now $\e=+$. Then, by \eqref{lem31-O-u}, we conclude that
    \begin{align*}
    \frac{(q^{m}-1)}{2}=d^{+}(q)\leq 3u^2+u\leq 4u^{2}<\frac{2(q^m-1)^3}{q^m(q^m+1)(q-1)},
    \end{align*}
    and so $q^m(q^m+1)(q-1)<4(q^m-1)^{2}$. Therefore,  $(q-1)<4(q^m-1)^{2}/[q^m(q^m+1)]<4$, and hence $q=3$. In this case, $G$ has rank $3$ by \cite[Theorem 1.1]{a:Kantor-rank3}, and by \cite{a:Dempwolff2001}, we know that there is no such symmetric design with $\lambda$ prime.
    
    Let now $\e=-$. Then \eqref{lem31-O-u} yields
    \begin{align*}
    \frac{(q^{m}+1)}{2}=d^{-}(q)\leq 3u^2-u< 3u^{2}<\frac{3(q^m+1)^3}{2q^m(q^m-1)(q+1)},
    \end{align*}
    and so $q^m(q^m-1)(q+1)<3(q^m+1)^{2}$. Since $q\geq 3$, it follows that $4q^m(q^m-1)<3(q^m+1)^{2}$, and so $q^{2m}<10q^{m}+3$, which is impossible as $m\geq 3$.
    
    Therefore, $i\geq 2$. Here by~\cite[Proposition 4.1.6]{b:KL-90}, $H\cap X$ is isomorphic to $\Omega_{i}^{\e}(q)\times \Omega_{n-i}(q){\cdot}4$, where $i$ is even and $\e\in\{+, -\}$. It follows from \cite[p.331]{a:Saxl2002},  we have that $v>q^{i(n-i)}/4$ and $k\leq 2a\lambda q^m$, where $n=2m+1$ and $m\geq 3$. Then by Lemma \ref{lem:six}(b), we have that
    \begin{align}\label{eq:eq2-A-ii-1-Oo}
    q^{i(n-i)}<16\lambda a^2q^{n-1}.
    \end{align}
    Since $\lambda$ is an odd prime divisor of $k$, Lemmas \ref{lem:New} and \ref{lem:six} imply that $\lambda$ must divide $a$, $p$, $q^{j_{1}}-\e$ or $q^{2j_{2}}-1$, where $j_{1}\leq \lfloor m/2 \rfloor$ and $j_{2}\leq \lfloor  (n-i-1)/2\rfloor$. Note that $m>i$. Thus $\lambda \leq (q^{(n-i-1)/2}+1)/2$, and so \eqref{eq:eq2-A-ii-1-Oo} implies that $q^{i(n-i)}<16a^2q^{(3n-i-3)/2}$. Therefore,
    \begin{align}\label{eq:A-ii-1-Oo-final}
    q^{n(2i-3)-2i^2+i+3}\leq 256 a^4.
    \end{align}
    As $m>i$, we have that $n>2i$, and so $q^{2i^2-5i+3}\leq 256 a^4$. This inequality holds only for $i=2$, in which case \eqref{eq:A-ii-1-Oo-final} implies that $q^{n-3}=q^{n(2i-3)-2i^2+i+3}\leq 256 a^4$. This inequality holds only for $(n,q)=(7, 3)$, in which case by \cite[Tables 8.39]{b:BHR-Max-Low}, $H\cap X$ is isomorphic to $\Omega_{2}^{-}(3)\times \Omega_{5}(3) \cdot 4$, and so \eqref{eq:v} implies that $v=22113$. By Lemmas \ref{lem:New} and \ref{lem:six}, $k$ is a divisor of $415720$. For these values of $(v, k)$, the fraction $k(k-1)/(v-1)$ is not prime, which is  a contradiction.\smallskip
    
    \noindent \textbf{(2.2)}  Let $X=\POm_{2m}^{\e}(q)$ with $q$ odd, $m\geq 4$ and $\e \in\{-, +\}$. Then $H=N_{i}$ with $i\leq m$. Set $n=2m$.
    
    If $i=1$, then by~\cite[Proposition 4.1.6]{b:KL-90}, $H\cap X$ is isomorphic to $\,^{\hat{}}\Omega_{2m-1}(q){\cdot}4$. Here by \eqref{eq:v}, we have that $v=q^{m-1}(q^{m}-\e1)/2$. Note that $|\Out(X)|$ divides $6a{\cdot}\gcd(4, q^{m}-1)$ and $\lambda$ is an odd prime divisor of $k$. Then by Lemmas \ref{lem:New} and \ref{lem:six}(b), $\lambda$ must divide $a$, $3$, $p$ or $q^{2j}-1$, for some $j\in\{1, 2, \ldots, m-1\}$. Therefore,
    \begin{align}\label{eq:lam-A-ii-1-Oo}
    \lambda \leq (q^{m-1}+1)/2.
    \end{align}
    According to \cite[p.332-333]{a:Saxl2002}, the parameter $k$ divides $\lambda (q^{m-1}+\e1)/2$ if $q\equiv 1 \mod{4}$, or $\lambda (q^{m-1}-\e1)/2$ if $q\equiv 3 \mod{3}$. Thus $k$ divides $\lambda d^{\e}(q)$, where $d^{\e}(q)=(q^{m-1}\pm \e1)/2$. By Lemma \ref{lem:six}(a), $k$ divides $\lambda(v-1)$. Therefore, $k$ must divide $\lambda \gcd(v-1, d^{\e}(q))$. Note that $\gcd(v-1, d^{ \e}(q))<q-1$. Therefore,  $k< \lambda(q-1)$. Then by \eqref{eq:lam-A-ii-1-Oo} and Lemma~\ref{lem:six}(b), we have that $\lambda q^{m-1}(q^{m}-\e1)\leq \lambda v<k^2<\lambda^2(q-1)^2$, and so \eqref{eq:lam-A-ii-1-Oo} implies that $q^{m-1}(q^{m}-\e1)<(q^{m-1}+1)(q-1)^2$, which is impossible.
    
    Therefore, we can assume that $1<i\leq m$. Then by \cite[p.19]{a:ABCD-PrimeRep}, $v>q^{i(2m-i)}/4$ and by \cite[p. 333]{a:Saxl2002}, $k\leq 4a\lambda {\cdot}q^m$. Then Lemma \ref{lem:six}(b) implies that $\lambda v<\lambda q^{i(2m-i)}<4k^2<64\lambda^2 a^2q^{2m}$. Thus
    \begin{align}\label{eq:eq-A-ii-2-Oo-1}
    q^{2m(i-1)-i^2}<64\lambda a^2.
    \end{align}
    Note that $\lambda$ is an odd prime divisor of $k$ and by Lemma \ref{lem:six}(b), $k$ divides $|\Out(X)|{\cdot}|H\cap X|$, where $|\Out(X)|$ divides $6a{\cdot}\gcd(q^{m}-\e1)$.  Here by~\cite[Proposition 4.1.6]{b:KL-90}, $H\cap X$ divides $|\Omega_{i}^{\delta_{1}}(q)\times \Omega_{n-i}^{\delta_{2}}(q){\cdot}4|$, where $\delta_{i}\in\{\circ, -, +\}$ and $i\geq 2$. Then $\lambda \leq \lambda_{i}(q)$, where $n=2m$ and
    \begin{align}\label{eq:lam-A-ii-Oo}
    \lambda_{i}(q)= \left\{
    \begin{array}{ll}
    2^{-1}{\cdot}(q^{(n-i)/2}+1),& \text{ if } (n-i)/2 \text{ is even and } \delta_{2}=-; \\
    2^{-1}{\cdot}(q^{(n-i-1)/2}+1),& \text{ if } i \text{ is odd};\\
    (q^{(n-i)/2}-(-1)^{\delta_{2}})(q-(-1)^{\delta_{2}})^{-1}, & \text{ otherwise}.
    \end{array}
    \right.
    \end{align}
    Thus by \eqref{eq:eq-A-ii-2-Oo-1} and \eqref{eq:lam-A-ii-Oo}, we have that $q^{2m(i-1)-i^2}<32a^2{\cdot}(q^{(n-i)/2}+1)$. Since $q^{(n-i)/2}+1<2q^{(n-i)/2}$, it follows that
    \begin{align}\label{eq:eq-A-ii-Oo-1-1}
    q^{m(2i-3)-i^2}<64a^2.
    \end{align}
    Note that $i\leq m$. Thus $q^{i^2-3i}\leq q^{m(2i-3)-i^2}<64a^2$, and so $q^{i^2-3i}<64a^2$. This inequality holds only for $i=2, 3$. If $i=3$, then by  \eqref{eq:eq-A-ii-2-Oo-1} and \eqref{eq:lam-A-ii-Oo}, we have that $q^{4m-9}<32a^2{\cdot}(q^{m-2}+1)$. Hence $q^{3m-7}<64a^2$. As $m\geq 4$, it follows that $q^{5}\leq q^{3m-7}<64a^2$, and so $q^5<64a^2$, which is impossible. Therefore, $i=2$. We now consider the following two subcases:\smallskip
    
    \begin{table}
        \centering
        \caption{Some large maximal subgroups of finite simple classical groups in Proposition~\ref{prop:orth}.}\label{tbl:class-s}
        \small
        \begin{tabular}{llll}
            \noalign{\smallskip}\hline\noalign{\smallskip}
            $X$ & $H\cap X$ & $v$ & $k$ divides \\
            \noalign{\smallskip}\hline\noalign{\smallskip}
            $\POm_{8}^{-}(3)$ & ${\hat{}}(\Omega_{2}^{-}(3)\times \Omega_{6}^{+}(3)){\cdot}2^2$ & $209223$ & $388177920$ \\
            $\POm_{8}^{-}(5)$ & ${\hat{}}(\Omega_{2}^{-}(5)\times \Omega_{6}^{+}(5)){\cdot}2^2$ & $102703125$ &$1392768000000$ \\
            $\POm_{8}^{-}(7)$ & ${\hat{}}(\Omega_{2}^{-}(7)\times \Omega_{6}^{+}(7)){\cdot}2^2$ & $6075747307$ &$296651671142400$ \\
            $\POm_{8}^{-}(9)$ & ${\hat{}}(\Omega_{2}^{-}(9)\times \Omega_{6}^{+}(9)){\cdot}2^2$ & $127287028233$ & $32486299582464000$ \\
            $\POm_{10}^{-}(3)$ & ${\hat{}}(\Omega_{2}^{-}(3)\times \Omega_{8}^{+}(3)){\cdot}2^2$ & $16409061$ & $158469754060800$ \\
            $\POm_{10}^{+}(3)$ & ${\hat{}}(\Omega_{2}^{+}(3)\times \Omega_{8}^{+}(3)){\cdot}2^2$ & $32549121$ & $158469754060800$ \\
            \noalign{\smallskip}\hline\noalign{\smallskip}
        \end{tabular}
    \end{table}

    \noindent{\textbf{(2.2.1)}} Let $m$ be even. If $\delta_{2}=-$, then by \eqref{eq:eq-A-ii-2-Oo-1} and \eqref{eq:lam-A-ii-Oo}, we have that $q^{2m-4}{\cdot}(q+1)<32a^2{\cdot}(q^{m-1}+1)$. Hence $q^{m-3}{\cdot}(q+1)<64a^2$. This inequality holds only for $(m, q)=(4, 3)$ in which case $v=189540$, which is not odd. If $\delta_{2}=+$, then by \eqref{eq:eq-A-ii-2-Oo-1} and \eqref{eq:lam-A-ii-Oo}, we have that $q^{2m-4}{\cdot}(q-1)<32a^2{\cdot}(q^{m-1}-1)$. Hence $q^{m-3}{\cdot}(q-1)<32a^2$. This inequality holds only for $(m, q)\in\{(4, 3), (4, 5), (4, 7), (4, 9)\}$.   We now apply~\cite[Proposition 4.1.6]{b:KL-90} and obtain $H\cap X$ as listed in Table~\ref{tbl:class-s}, and considering the fact that $v$ is odd, we have that $(m,q,\e)\in\{(4, 3,-), (4, 5,-), (4, 7,-), (4, 9,-)\}$. Moreover, Lemma \ref{lem:six}(b) says that $k$ divides $|\Out(X)|{\cdot}|H\cap X|$, and so we can find the possible values of $k$ as in the fourth column of Table~\ref{tbl:class-s}.  This is a contradiction as for each $k$ and $v$ as in Table~\ref{tbl:class-s}, the fraction $k(k-1)/(v-1)$ is not prime.\smallskip
    
    \noindent{\textbf{(2.2.2)}} Let $m$ be odd. If $\delta_{2}=-$, then by \eqref{eq:eq-A-ii-2-Oo-1} and \eqref{eq:lam-A-ii-Oo}, we have that $q^{2m-4}<32a^2{\cdot}(q^{m-1}+1)$. Hence $q^{m-3}<64a^2$. This inequality holds only for $m=5$ and $q=3,5,7,9$. All these cases can be ruled out as $v$ has to be odd. If $\delta_{2}=+$, then by \eqref{eq:lam-A-ii-Oo} and \eqref{eq:eq-A-ii-2-Oo-1}, we have that $q^{2m-4}{\cdot}(q-1)<32a^2{\cdot}(q^{m-1}-1)$. Hence $q^{m-3}{\cdot}(q-1)<32a^2$. This inequality holds only for $(m, q)=(5, 3)$ for $\e=\pm$. By~\cite[Proposition 4.1.6]{b:KL-90}, we can obtain $H\cap X$ as in Table~\ref{tbl:class-s}, and for each such $H\cap X$, by \eqref{eq:v}, we find $v$ as in the third column of Table~\ref{tbl:class-s}. Note by Lemma \ref{lem:six}(b) that $k$ divides $|\Out(X)|{\cdot}|H\cap X|$, and so we can find the possible values of $k$ as in the fourth column of Table~\ref{tbl:class-s}. All these cases can be ruled out as the fraction $k(k-1)/(v-1)$ is not prime.\smallskip

    \noindent \textbf{(3)} Let $H_{0}$ be the stabilizer of an orthogonal decomposition $V=\oplus V_{j}$ with all $V_{j}$'s isometric, and let $q$ be odd. This case has to be treated separately for both odd and even dimension of $V$.\smallskip
    
    \noindent \textbf{(3.1)} Let $X=\Omega_{2m+1}(q)$ with $q$ odd and $m\geq 3$. In this case $H$ is the stabilizer of a subspace decomposition into
    isometric non-singular spaces of dimension $i$, where $i$ is odd.
    
    Let $i=1$. Then by~\cite[Proposition 4.2.15]{b:KL-90}, the subgroup $H\cap X$ is isomorphic to $2^{2m}{\cdot}\S_{2m+1}$ or $2^{2m}{\cdot}\A_{2m+1}$ if $q\equiv \pm 1 \mod{8}$ or $q\equiv \pm3 \mod{8}$, respectively. The subgroups $H\cap X$ satisfying $|X|<|\Out(X)|^2{\cdot}|H\cap X|^3$ are listed in Table~\ref{tbl:class-s-2-2}, and for each such $H\cap X$, by \eqref{eq:v}, we obtain the parameter $v$ as in the fourth column of Table~\ref{tbl:class-s-2-2}. Moreover, Lemma \ref{lem:six}(b) says that $k$ divides $|\Out(X)|\cdot |H\cap X|$, and so we can find the possible values of $k$ as in the fifth column of Table~\ref{tbl:class-s-2-2}. For each possible case, we observe that $k(k-1)/(v-1)$ is not prime, which is a contradiction.
    \begin{table}
        \small
        \centering
        \caption{Some parameters for Case 3 in Proposition~\ref{prop:orth} }\label{tbl:class-s-2-2}
        \begin{tabular}{cllll}
            \noalign{\smallskip}\hline\noalign{\smallskip}
            Line &
            $X$ &
            $H\cap X$ &
            $v$ &
            $k$ divides
            \\
            \noalign{\smallskip}\hline\noalign{\smallskip}
            $1$ &
            $\Omega_{7}(3)$ &
            $2^{6}{\cdot}\A_{7}$ &
            $28431$ &
            $645120$
            \\
            $2$ &
            $\Omega_{7}(5)$ &
            $2^{6}{\cdot}\A_{7}$ &
            $  1416796875  $   &
            $  645120 $
            
            \\
            $3$ &
            $\Omega_{9}(3)$ &
            $2^{8}{\cdot}\A_{9}$ &
            $  1416290265  $   &
            $  185794560 $
            \\
            $4$ &
            $\Omega_{11}(3)$ &
            $2^{10}{\cdot}\A_{11}$ &
            $  3741072100580529  $   &
            $81749606400$
            \\
            $5$ &
            $\Omega_{13}(3)$ &
            $2^{12}{\cdot}\A_{13}$ &
            $  564416277323644023433155  $   &
            $  51011754393600 $
            \\
            $6$ &
            $\POm_{8}^{+}(3)$ &
            $2^{6}{\cdot}\A_{8}$ &
            $  3838185 $ &
            $  15482880  $
            \\
            $7$ &
            $\POm_{8}^{+}(5)$ &
            $2^{6}{\cdot}\A_{8}$ &
            $  6906884765625 $ &
            $  15482880  $
            \\
            $8$ &
            $\POm_{12}^{+}(3)$ &
            $2^{10}{\cdot}\A_{12}$ &
            $  27575442453379079259 $ &
            $  2942985830400  $
            \\
            $9$ &
            $\POm_{10}^{-}(3)$ &
            $2^{8}{\cdot}\A_{10}$ &
            $  1399578039873 $ &
            $  3715891200  $
            \\
            $6$ &
            $\POm_{14}^{-}(3)$ &
            $2^{12}{\cdot}\A_{14}$ &
            $  32152618284915465959467883895 $ &
            $  1428329123020800  $
            \\
            \noalign{\smallskip}\hline\noalign{\smallskip}
        \end{tabular}
    \end{table}
    
    Therefore, $i\geq 3$, and hence~\cite[Proposition 4.2.14]{b:KL-90} implies that $H\cap X$ is isomorphic to
    \begin{align*}
    (2^{t-1}\times \Omega_{i}(q)^t{\cdot}2^{t-1}){\cdot}\S_{t},
    \end{align*}
    where $it=2m+1$.
    
    Let $i=3$. Then $H\cap X$ is isomorphic to $(2^{t-1}\times \Omega_{3}(q)^t\cdot 2^{t-1})\cdot \S_{t}$, and so  by Lemma \ref{lem:up-lo-b}, we conclude that $q^{m^2}\prod_{j=1}^{m}(q^{2j}-1)<a^2\cdot 2^{3}\cdot 2^{6t-6}\cdot (t!)^3\cdot q^{3t}(q^2-1)^{3t}/2^{-3t}$. Since $a^2<q$ and $q^{m^2}\leq \prod_{j=1}^{m}(q^{2j}-1)$, it follows that $q^{2m^2}<2^{3t-3}\cdot (t!)^3\cdot q^{9t+1}$. Thus $q^{2m^2-9t-1}<2^{3t-3}\cdot (t!)^3$. Since $2m+1=3t$, we conclude that
    \begin{align}\label{eq:A-iii-eq-O-1}
    q^{9t^2-24t-1}<2^{6t-6}\cdot (t!)^6.
    \end{align}
    If $t=3$, then $q^{8}<2^{18}\cdot 3^6$. This inequality holds for $q\in\{3,5,7,9\}$, and so in each case, we easily observe by \eqref{eq:v} that $v$ is even, which is a contradiction. Thus $t\geq 5$. Since by Lemma \ref{lem:factor-bound}(a) we have that $t!<5^{(t^{2}-3t+1)/3}$, it follows from  \eqref{eq:A-iii-eq-O-1} that  $q^{9t^2-24t-1}<2^{6t-6}\cdot (t!)^6<2^{6t-6}\cdot 5^{2t^2-6t+2}$. Thus $q^{9t^2-24t-1}<2^{6t-6}\cdot 5^{2t^2-6t+2}$. Since $2^{6t-6}\cdot 5^{2t^2-6t+2}<5^{2t^2}$, it follows that $q^{9t^2-24t-1}<5^{2t^2}$. Then $\log_{p}q \cdot (9t^2-24t-1)<\log_{p}5 \cdot (2t^2)<3t^2$, and so $9t^2-24t-1<3t^2$. Thus, $6t^2-24t-1<0$, this inequality does not hold for any $t\geq 5$, which is a contradiction.
    
    Let $i\geq 5$. Then by Corollary \ref{cor:large} and Lemma \ref{lem:up-lo-b}, we have that $q^{m^2}\prod_{j=1}^{m}(q^{2j}-1)<a^2{\cdot}2^{3}{\cdot}2^{6t-6}\cdot (t!)^3{\cdot}q^{3it(i-1)/2}$. Since $a^2<q$ and $q^{m^2}\leq \prod_{j=1}^{m}(q^{2j}-1)$, it follows that $q^{2m^2}<2^{6t-3}\cdot (t!)^3{\cdot}q^{[3it(i-1)+2]/2}$. Thus $q^{2m^2-[3it(i-1)+2]/2}<2^{6t-3}\cdot (t!)^3$. Since $2m+1=it$, we conclude that
    \begin{align}\label{eq:A-iii-eq-O-2}
    q^{(it-1)^2-3it(i-1)-2}<2^{12t-6}{\cdot}(t!)^6.
    \end{align}
    If $t=3$, then $q^{3i-1}<2^{36}{\cdot}3^6$. Since $2^{36}{\cdot}3^6<3^{29}$, it follows that $q^{3i-1}<3^{29}$. This inequality holds only for $i\in\{5, 7, 9\}$. If $i=5$, then by \eqref{eq:A-iii-eq-O-2}, we conclude that
    $q^{14}<2^{36}{\cdot}3^6$. This inequality holds only for $q\in\{3,5,7,9\}$. Then by \eqref{eq:v}, we easily observe that $v$ is not odd, which is a contradiction. By the same manner, we can rule out the remaining case where $i=7, 9$.  Therefore $t\geq 5$, and hence by Lemma \ref{lem:factor-bound}(a), we have that $t!<5^{(t^{2}-3t+1)/3}$, and so \eqref{eq:A-iii-eq-O-2} implies that  $q^{(it-1)^2-3it(i-1)-2}<2^{12t-6}{\cdot}(t!)^6<2^{12t-6}{\cdot}5^{2t^2-6t+2}$. Since $2^{12t-6}{\cdot}5^{2t^2-6t+2}<5^{2t^2}$, it follows that $q^{(it-1)^2-3it(i-1)-2}<5^{2t^2-1}$. Then $[i^2t(t-3)+it-1]\cdot \log_{p}q<(2t^2-1)\cdot \log_{p} 5<3t^2$, and so
    \begin{align}\label{eq:A-ii-2-O-2-4.2.1}
    i^2t(t-3)+it-1<3t^2.
    \end{align}
    Note that $i\geq 5$. Then \eqref{eq:A-ii-2-O-2-4.2.1} implies that $25t^2-20t-1\leq i^2t(t-3)-5it-1<3t^2$, and so $23t^2-20t-1<0$, which is impossible.\smallskip
    
    \noindent \textbf{(3.2)} Let $X=\POm_{2m}^{\e}(q)$ with $q$ odd, $m\geq 4$ and $\e \in\{-, +\}$. In this case, $H$ is an imprimitive subgroup of $G$ stabilizing a
    decomposition $V=V_{1} \oplus \cdots \oplus V_{t}$ with the dimension of each $V_j$'s equal to $i$, so $2m=it$.\smallskip
    
    \noindent \textbf{(3.2.1)} Let $i=1$. Then by Corollary \ref{cor:large} and~\cite[Proposition 4.2.15]{b:KL-90}, we can obtain $H\cap X$ as listed in Table~\ref{tbl:class-s-2-2}. For each such $H\cap X$, by \eqref{eq:v}, we can obtain $v$ as in the third column of Table \ref{tbl:class-s-2-2}. Moreover, Lemma \ref{lem:six}(b) says that $k$ divides $|\Out(X)|\cdot |H\cap X|$, and so we can find the possible values of $k$ as in the fourth column of Table~\ref{tbl:class-s-2-2}.  This is a contradiction as for each $k$ and $v$ as in Table \ref{tbl:class-s-2-2}, the fraction $k(k-1)/(v-1)$ is not prime. Hence $i\geq 2$.\smallskip
    
    \noindent \textbf{(3.2.2)} Let $i$ be odd. Then by \cite[Proposition 4.2.14]{b:KL-90}, $H\cap X$ is isomorphic to $(2^{t-2}\times \Omega_{i}(q)^t{.}2^{t-1}){.}\S_{t}$ with $t$ even and $\e=(-1)^{m(q-1)/2}$.
    
    If $i=3$, then $t\geq 4$ as $3t=it=2m\geq 8$.  It follows from Corollary \ref{cor:large} and Lemma \ref{lem:up-lo-b} that $q^{m(2m-1)}<|\Out(X)|^2\cdot 2^{6t-6}\cdot (t!)^3\cdot q^{3t}(q^2-1)^{3t}/2^{-3t}$. Since $|\Out(X)|$ divides $24a$ and $a^2<q$, we conclude that $q^{2m^2-m}<2^{3t}\cdot 3^2\cdot (t!)^3 q^{9t+1}$. Thus $q^{2m^2-m-9t-4}<2^{3t} \cdot (t!)^3$. Since $2m=3t$, we have that
    \begin{align}\label{eq:A-ii-O+-5.1.2.1}
    q^{9t^2-21t-8}<2^{6t}\cdot (t!)^6.
    \end{align}
    Note by Lemma \ref{lem:factor-bound}(b) that $t!<2^{4t(t-3)/3}$. Thus \eqref{eq:A-ii-O+-5.1.2.1} implies that $q^{9t^2-21t-8}<2^{6t}\cdot (t!)^6<2^{6t}\cdot 2^{8t^2-24t}$, and so $q^{9t^2-21t-8}<2^{8t^2-18t}$. Then $(9t^2-21t-8)\cdot \log_{p}q<(8t^2-18t)\cdot \log_{p}2<(8t^2-18t)\times 0.7$, and so $90t^2-210t-80<56t^2-126t$. Therefore, $34t^2-84t-80<0$, this inequality does not hold for any $t\geq 4$, which is a contradiction.
    
    Therefore, $i\geq 5$. If $t=2$, then $m=i$ as $2m=it$. Let $u$ be a positive integer such that $i=2u+1$. Then by \eqref{eq:v}, we have that
    \begin{align*}
    v=\frac{q^{3u^2+2u}(q^{2u+1}-\e1)(q^{4u}-1)(q^{4u-2}-1)\cdots(q^2-1)}{2(q^{2u}-1)^2(q^{2u-2}-1)^2\cdots(q^2-1)^2},
    \end{align*}
    which is even, and this is a contradiction. If $t\geq 4$, then by Corollary \ref{cor:large} and Lemma \ref{lem:up-lo-b}, we have that $q^{m(2m-1)}<|\Out(X)|^2\cdot 2^{6t-6}\cdot  (t!)^3\cdot q^{3it(i-1)/2}$. Thus, $q^{4m^2-2m}<|\Out(X)|^4\cdot 2^{12t-12}\cdot  (t!)^6\cdot q^{3it(i-1)}$. Note that $|\Out(X)|$ divides $24a$ and $a^2<q$. Thus, $q^{4m^2-2m}<2^{12t}\cdot 3^{4}\cdot (t!)^6q^{3it(i-1)+2}$. Since $2m=it$, we conclude that
    \begin{align}\label{eq:A-iii-eq-O+-5.1.2.2}
    q^{i^2t^2-it-3it(i-1)-6}<2^{12t}\cdot (t!)^6.
    \end{align}
    Since $t!<2^{4t(t-3)/3}$ for $t\geq 4$ by Lemma \ref{lem:factor-bound}(b), we conclude that $q^{i^2t^2-it-3it(i-1)-6}<2^{12t}\cdot (t!)^6<2^{12t}\cdot 2^{8t^2-24t}$, and so $q^{i^2t^2-it-3it(i-1)-6}<2^{8t^2-12t}$. Then $(i^2t^2-it-3it(i-1)-6)\cdot \log_{p}q <(8t^2-12t)\cdot \log_{p}2<(8t^2-12t)\times 0.7$, and so $10i^2t(t-3)+20it-60<56t^2-84t$. Since $i\geq 5$, it follows that $250t^2-650t-60\leq 10i^2t(t-3)+20it-60<56t^2-84t$, Thus, $250t^2-650t-60<56t^2-84t$, and so $194t^2-566t-60<0$, this inequality does not hold for any $t\geq 4$, which is a contradiction.\smallskip
    
    \noindent \textbf{(3.2.3)} Let $i$ be even. Then by~\cite[Proposition 4.2.11]{b:KL-90}, the $H\cap X$ is isomorphic to
    \begin{align*}
    d^{-1}\Omega_{i}^{\e_{1}}(q)^t{.}2^{2(t-1)}{.}\S_{t},
    \end{align*}
    where $\e=\e_{1}^{t}$ and $d\in\{1, 2, 4\}$.

    If $t=2$, then $m=i$, as $2m=it$. Let $u$ be a positive integer such that $i=2u$. Then by \eqref{eq:v}, we have that
    \begin{align*}
    v=\frac{q^{2u^2}(q^{u}+\e_{1}1)\cdot (q^{4u-2}-1)(q^{4u-4}-1)\cdots(q^2-1)}{2\cdot (q^{2u-2}-1)^2(q^{2u-4}-1)^2\cdots(q^2-1)^2}.
    \end{align*}
    This contradicts the fact that $v$ is odd. Therefore, $t\geq 3$.
    \begin{table}
        \small
        \caption{Some parameters for Case 3.2.3 in Proposition~\ref{prop:orth} }\label{tbl:class-s-2-3}
        \begin{tabular}{cllll}
            \noalign{\smallskip}\hline\noalign{\smallskip}
            Line &
            $X$ &
            $H\cap X$ &
            $v$ &
            $k$ divides
            \\
            \noalign{\smallskip}\hline\noalign{\smallskip}
            $1$ &
            $\POm_{12}^{+}(3)$ &
            $\,{\hat{}}\Omega_{4}^{+}(3)^3{.}2^2{.}\Sym_{3}$ &
            $5898080746972747508175$ &
            $18345885696$
            \\
            $2$ &
            $\POm_{12}^{+}(5)$ &
            $\,{\hat{}}\Omega_{4}^{+}(5)^3{.}2^2{.}\Sym_{3}$ &
            $181234396436428964138031005859375$   &
            $286654464000000$
            \\
            \noalign{\smallskip}\hline\noalign{\smallskip}
        \end{tabular}
    \end{table}
    
    If $i=2$, then $m=t$, and so by \eqref{eq:v} and Lemma \ref{lem:up-lo-b}, we have that $v>q^{2t^2-t}/[2^{t-2}(t!)\cdot (q+1)^{t}]$. By \cite[p. 333]{a:Saxl2002}, the parameter $k$ is at most $2^5\cdot 3\cdot \lambda a\cdot t(t-1)(q+1)^2$, and so by Lemma \ref{lem:six}(b), we conclude that $\lambda q^{2t^2-t}/[2^{t-2}(t!)\cdot (q+1)^{t}]<\lambda v< k^2\leq 2^{10}3^2\lambda^2 a^2\cdot t^2(t-1)^2(q+1)^4$. Since $a^2<q$, $2^{10}3^2<2^{14}$ and $t^2(t-1)^2<t^4$, it follows that
    \begin{align}\label{eq:lam-A-iii-O+--2}
    q^{2t^2-t-1}<2^{t+12}\lambda t^4(t!)(q+1)^{t+4}.
    \end{align}
    Note that $\lambda$ is a prime divisor of $k$. Thus by Lemma \ref{lem:six}(b), $\lambda$ must divide $a$, $t$ or $(q+\e1)/2$. Therefore $\lambda \leq \max\{a, t, (q+1)/2\}<t(q+1)/2$, and so by \eqref{eq:lam-A-iii-O+--2}, we have that
    \begin{align}\label{eq:lam-A-iii-O+---2-1}
    q^{2t^2-t-1}<2^{t+11}t^5(t!)(q+1)^{t+5}.
    \end{align}
    As $q+1<2q$ and $t^5\leq 2^{3t}$, we conclude that $q^{2t^2-2t-6}<2^{5t+15}(t!)$. Note that $t=m\geq 4$. Then by Lemma \ref{lem:factor-bound}(b), we have that $q^{2t^2-2t-6}<2^{5t+15}(t!)<2^{5t+15}2^{4t(4-3)/3}$. Thus $q^{6t^2-6t-18}<2^{4t^2+3t+45}$, and so $(6t^2-6t-18)\cdot \log_{p}q<(4t^2+3t+45)\cdot \log_{p}2\leq (4t^2+3t+45)\times 0.7$. Hence $60t^2-60t-180<28t^2+21t+300$, and so $32t^2-81t-480<0$, then $t=4, 5$. If $t=5$, then \eqref{eq:lam-A-iii-O+---2-1} implies that $q^{44}<2^{16}\cdot 5^5(5!)\cdot (q+1)^{10}<2^{26}\cdot 5^5(5!)q^{10}$, and so $q^{34}<2^{26}\cdot 5^5\cdot (5!)$, which is impossible. If $t=4$, then by the same manner, we must have $q^{18}<2^{37}\cdot 3$, which is valid for $q=3$. Since $\lambda$ divides $a=1$, $t=4$ or $(q+\e1)/2=(3+\e1)/2$, we conclude that $\lambda=2$, which is a contradiction.
    
    If $i=4$, then $m=2t$, and so by Corollary \ref{cor:large} and Lemma \ref{lem:up-lo-b}, we conclude that $q^{m(2m-1)}<|\Out(X)|^2\cdot 2^{6t-3}\cdot  (t!)^3\cdot q^{6t}(q^4-1)^{3t}$. Thus, $q^{m(2m-1)}<|\Out(X)|^2\cdot 2^{6t-3}\cdot  (t!)^3\cdot q^{18t}$. Note that $|\Out(X)|$ divides $24a$, $a^2<q$ and $m=2t$. Thus, $q^{8t^2-2t}=q^{m(2m-1)}<2^{6t+3}3^{2}(t!)^3q^{18t+1}$, and so
    \begin{align}\label{eq:A-iii-eq-O+-5.1.2.3}
    q^{8t^2-20t-3}<2^{6t-3}\cdot  (t!)^3.
    \end{align}
    If $t=3$, then \eqref{eq:A-iii-eq-O+-5.1.2.3} yields $q^{9}<2^{18}\cdot 3^3$, and so $q=3, 5$. In each of these cases, $H\cap X$ and $v$ are recorded as in Table~\ref{tbl:class-s-2-3}. By Lemma \ref{lem:six}(b), the parameter $k$ divides $|\Out(X)|\cdot |H\cap X|$ as in the fifth column of Table~\ref{tbl:class-s-2-3}. It is easy to check for each possible parameters $v$ and $k$ that the fraction $k(k-1)/(v-1)$ is not prime, which is a contradiction.
    If $t\geq 4$, then by Lemma \ref{lem:factor-bound}(b), we have that $t!<2^{4t(t-3)/3}$, and so \eqref{eq:A-iii-eq-O+-5.1.2.3} implies that $q^{8t^2-20t-3}<2^{6t-3}\cdot  (t!)^3<2^{4t^2-6t-3}$. Thus $q^{8t^2-20t-3}<2^{4t^2-6t-3}$, and so $(8t^2-20t-3)\cdot \log_{p}q<(4t^2-6t-3)\cdot \log_{p}2\leq (4t^2-6t-3)\times 0.7$. Hence $80t^2-200t-30<28t^2-42t-21$, and so $52t^2-158t-9<0$, which has no solution for $t\geq 4$, which is a contradiction.
    
    If $i\geq 6$, then Corollary \ref{cor:large} and Lemma \ref{lem:up-lo-b} imply that $q^{m(2m-1)}<|\Out(X)|^2\cdot 2^{6t-3}\cdot  (t!)^3\cdot q^{3it(i-1)/2}$. Thus, $q^{4m^2-2m}<|\Out(X)|^4\cdot 2^{12t-6}\cdot  (t!)^6\cdot q^{3it(i-1)}$. Note that $|\Out(X)|$ divides $24a$ and $a^2<q$. Thus, $q^{4m^2-2m}<2^{12t+6}3^{4}(t!)^6q^{3it(i-1)+2}$. Since $2m=it$, we conclude that
    \begin{align}\label{eq:A-iii-eq-O+-5.1.2.2-1}
    q^{i^2t(t-3)+2it-6}<2^{12t+6}\cdot (t!)^6.
    \end{align}
    If $t=3$, then \eqref{eq:A-iii-eq-O+-5.1.2.2-1} yields $q^{6i-6}<2^{48}\cdot 3^6$. As $q$ is odd and $i\geq 6$, it follows that $q^{6i-6}<2^{48}\cdot 3^6$. This inequality holds only for $(i,q)=(6,3)$ in which case by \eqref{eq:v}, we easily observe that $v$ is even, which is a contradiction. If $t\geq 4$, then $t!<2^{4t(t-3)/3}$ by  Lemma \ref{lem:factor-bound}(b). Thus by \eqref{eq:A-iii-eq-O+-5.1.2.2-1}, we conclude that $q^{i^2t(t-3)+2it-6}<2^{12t+6}(t!)^6<2^{12t+6}2^{8t^2-24t}$, and so $q^{i^2t(t-3)+2it-6}<2^{8t^2-12t+6}$. Then $[i^2t(t-3)+2it-6]\times \log_{p}q<(8t^2-12t+6)\times \log_{p}2\leq (8t^2-12t+6)\times  0.7$. Hence $10i^2t(t-3)+20it-60<56t^2-84t+42$. Since $i\geq 6$, it follows that $360t^2-960t-60\leq 10i^2t(t-3)+20it-60<56t^2-84t+42$ and so $304t^2-876t-102<0$, which is impossible for  $t\geq 4$.\smallskip

    \noindent \textbf{(4)-(5)} In these cases, the pairs $(X,H_{0})$ are recorded in Table~\ref{tbl:A4-i-4}, and for each case, by \eqref{eq:v}, we obtain the parameter $v$ as in the fifth column of the same table. Moreover, for each pairs $(X,H\cap X)$, by Lemmas \ref{lem:New} and \ref{lem:six}(b), the parameter $k$  divides the number listed in the sixth column of Table~\ref{tbl:A4-i-4}. We now apply Lemma~\ref{lem:six}(b), and so $v<k^{2}$. For each row, this inequality is true only for $q$ given in the last column of Table~\ref{tbl:A4-i-4}. It is easy to check that for each appropriate pairs $(v,k)$, the fraction $k(k-1)/(v-1)$ is not a prime number.\smallskip
    \begin{table}
        \small
        \centering
        \caption{Some parameters for Cases 4 and 5 in Proposition~\ref{prop:orth}}\label{tbl:A4-i-4}
        \begin{tabular}{cllllll}
            \noalign{\smallskip}\hline\noalign{\smallskip}
            Line &
            $X$ &
            $H_{0}$ &
            Conditions &
            $v$ &
            $k$ divides &
            $q$
            \\
            \noalign{\smallskip}\hline\noalign{\smallskip}
            $1$ &
            $\Omega_{7}(q)$ &
            $\SO_{7}(2)$ &
            $q=p \equiv \pm 3 \mod{8}$ &
            $\frac{q^{9}(q^{6}-1)(q^4-1)(q^2-1)}{2^{10}{\cdot}3^4{\cdot}5{\cdot}7}$ &
            $2^{10}{\cdot}3^4{\cdot}5{\cdot}7$ &
            $3$, $5$
            \\
            $2$ &
            $\POm_{8}^{+}(q)$ &
            $\Omega_{8}^{+}(2)$ &
            $q=p \equiv \pm 3 \mod{8}$ &
            $\frac{q^{12}(q^{6}-1)(q^4-1)^2(q^2-1)}{2^{14}{\cdot}3^5{\cdot}5^2{\cdot}7}$ &
            $2^{15}{\cdot}3^5{\cdot}5^2{\cdot}7$&
            $3$, $5$
            \\
            $3$ &
            $\Omega_{8}^{+}(q)$ &
            $2^3\cdot 2^6\cdot \PSL_{3}(2)$ &
            $q=p \equiv \pm 3 \mod{8}$&
            $\frac{q^{12}(q^{6}-1)(q^4-1)^2(q^2-1)}{2^{14}{\cdot}3{\cdot}7}$&
            $2^{15}{\cdot}3^{2}{\cdot}7$&
            $3$
            \\
            \noalign{\smallskip}\hline\noalign{\smallskip}
        \end{tabular}
    \end{table}
    
    \noindent \textbf{(6)} Let $H = N_{G}(X(q_{0}))$ with $q=q_{0}^t$ odd and $t$ odd prime. Here, as noted before, we only need consider the case where $X=\POm_{2m}^{\e}(q)$ with $q$ odd, $n=2m$ and $\e=\pm$. By~\cite[Proposition 4.5.10]{b:KL-90}, the subgroup $H_{0}$ is isomorphic to $\POm_{2m}^{\e}(q_{0})$, where $m\geq 4$. Note that $|\Out(X)|$ divides $6a$. Then by Lemma~\ref{lem:up-lo-b} and the inequality $|X|<|\Out(X)|^2{\cdot}|H\cap X|^3$, we have that $q_{0}^{tm(2m-1)}<2^5{\cdot}3^2{\cdot}a^2{\cdot}q_{0}^{3m(2m-1)}{\cdot}(1+q_{0}^{-m})^3$. Since $a^2<2q$ and $1+q_{0}^{-m}<2$, it follows that $q_{0}^{(2m^2-m)(t-3)-t}<2^{9}{\cdot} 3^{2}$. If $t\geq 5$, then $q_{0}^{4m^2-2m-5}\leq q_{0}^{(2m^2-m)(t-3)-t}<2^{9}{\cdot} 3^{2}$, which is impossible. Hence $t=3$ in which case by \eqref{eq:v} and Lemma \ref{lem:up-lo-b}, we have that $v>q_{0}^{4m^2-2m-4}$. Since $k$ divides $\lambda (v-1,|H|)$ and $v-1$ is coprime to $q_{0}$, the parameter $k$ must divide $6a\lambda{\cdot}|H\cap X|_{p'}$. Since $|H\cap X|_{p'}< q_{0}^{2m(m-1)}(q_{0}^{m}+1)^2$, Lemma~\ref{lem:six}(b) implies that $\lambda q_{0}^{4m^2-2m-4}< \lambda v <k^{2}\leq 36a^2\lambda^2q_{0}^{2m(m-1)}(q_{0}^{m}+1)^2$. Therefore,
    \begin{align}\label{eq:k-A-i-Oe}
    q_{0}^{2m^2-4}<36a^2\lambda {\cdot}(q_{0}^{m}+1)^2.
    \end{align}
    Note that $\lambda$ is an odd prime divisor of $k$. Thus Lemmas \ref{lem:New} and \ref{lem:six}(b) imply that $\lambda$ divides $3$, $a$, $p$, $q_{0}^{m}-\e1$ or $q_{0}^{2j}-1$, for some $j\in \{1, \ldots, m-1\}$, and so $\lambda \leq q_{0}^m+1$. Then by inequality \eqref{eq:k-A-i-Oe}, we have that $q_{0}^{2m^2-4}<36a^2{\cdot}(q_{0}^m+1)^3$. Since $q_{0}^m+1<2q_{0}$,  $q_{0}^{2m^2-3m-4}<2^5{\cdot}3^2a^2$. As $a=3s$, $m\geq 4$ and $q_{0}$ is odd, $3^{16s}\leq q_{0}^{2m^2-3m-4}<2^5{\cdot}3^3s^2$, and so $3^{16s}<2^5{\cdot}3^3s^2$, which is impossible.
\end{proof}

\noindent{\textbf{Proof of Theorem~\ref{thm:main}}}
Suppose that $\Dmc$ is a nontrivial symmetric $(v, k, \lambda)$ design admitting a flag-transitive and point-primitive automorphism group $G$ with socle $X$ a finite simple group of Lie type. Suppose also that $\lambda$ is prime. The symmetric designs with $\lambda=2,3$ admitting flag-transitive transitive automorphism groups are classified in \cite{a:Zhou-lam3-affine,t:Regueiro,a:Regueiro-classification}, and so by a quick check, we observe that the pairs $(\Dmc, G)$ are as in Table \ref{tbl:main}. Therefore, we can assume that $\lambda\geq 5$. Since $k(k-1)=\lambda(v-1)$, it follows that $\lambda$ is coprime to $k$ or $\lambda$ divides $k$. In the former case, by \cite[Corollary~1.2]{a:ADM-AS-CP}, we conclude that $\Dmc$ is a projective space $\PG_{n}(q)$ or $\Dmc$ is the unique Hadamard design with parameters $(11,5,3)$ which has been already recorded in Table \ref{tbl:main}. We now consider the latter case where $\lambda$ divides $k$. We first observe by \cite[Corollary~1.2]{a:ABD-Exp} that the socle $X$ cannot be a finite simple exceptional group. Let now $X$ be a finite simple classical groups. Since $G$ is point-primitive, the point-stabiliser $H=G_\alpha$ is maximal in $G$, and considering the fact that $k(k-1)=\lambda(v-1)$, we conclude that $v$ is odd, and our main result then follows from Propositions~\ref{prop:psl}-\ref{prop:orth}.\smallskip

\noindent{\textbf{Proof of Corollary~\ref{thm:main}}}
Suppose that $\Dmc$ is a nontrivial symmetric $(v, k, \lambda)$ design with $\lambda$ prime admitting a flag-transitive and point-imprimitive automorphism group $G$.  Suppose also that $(c,d,l)$ is as in the statement of Corollary~\ref{thm:main}. If $(v,k,\lambda)$ is not one of the possibilities mentioned in Corollary~\ref{thm:main}, then \cite[Theorem~1.1]{a:Praeger-imprimitive} implies that $k=\lambda^2/2$, and since $\lambda$ is prime, we conclude that $\lambda=2$, and hence $k=4/2=2=\lambda$, which is a contradiction.

\subsection*{Acknowledgements}

The authors are grateful to the anonymous referees for helpful and constructive comments.




\end{document}